\documentclass[letterpaper,11pt]{article}

\usepackage{amsmath,amsfonts,amsthm,amssymb,graphicx,paralist,mathtools,algorithm,multirow}
\usepackage{fullpage,algpseudocode}
\usepackage[dvipsnames]{xcolor}

\DeclareMathOperator{\dom}{dom}

\newcommand{\R}{{\mathbf R}}
\newcommand{\E}{{\mathbf E}}
\newcommand{\Exp}{{\mathbf E}}
\newcommand{\Prob}{{\mathbf P}}

\newcommand{\bp}{\begin{proof}[\textbf{Solution}]}
\newcommand{\ep}{\end{proof}}

\newcommand{\eqdef}{\overset{\text{def}}{=}}
\newcommand{\xbi}{x^{(i)}}

\newcommand{\ii}{^{(i)}}

\newcommand{\lmax}{L'}

\newcommand{\ve}[2]{\langle #1 , #2 \rangle}

\newtheorem{theorem}{Theorem}
\newtheorem{lemma}[theorem]{Lemma}
\newtheorem{proposition}[theorem]{Proposition}

\newtheorem{definition}[theorem]{Definition}

\begin{document}

\title{Separable Approximations and Decomposition Methods\\  for the Augmented Lagrangian}

\author{Rachael Tappenden \qquad Peter Richt\'{a}rik \qquad  Burak B\"{u}ke \footnote{All authors: James Clerk Maxwell Building, School of Mathematics, The University of Edinburgh, United Kingdom. The work of all three authors was supported by the EPSRC grant EP/I017127/1 (Mathematics for Vast Digital Resources). The work of PR and RT was also partially supported by the Centre for Numerical Algorithms and Intelligent Software (funded by EPSRC grant EP/G036136/1 and the Scottish Funding Council).}}

\date{August 30, 2013}

\maketitle

\begin{abstract}In this paper we study decomposition methods based on separable approximations for minimizing the augmented Lagrangian. In particular, we study and compare the  Diagonal Quadratic Approximation Method (DQAM)  of Mulvey and Ruszczy\'{n}ski \cite{Mulvey92} and the Parallel Coordinate Descent Method (PCDM) of Richt\'{a}rik and Tak\'{a}\v{c} \cite{Richtarik12a}. We show that the two methods are equivalent for feasibility problems up to the selection of a single step-size parameter. Furthermore, we prove an improved complexity  bound for PCDM under strong convexity, and show that this bound is at least $8(L'/\bar{L})(\omega-1)^2$ times better than the best known bound for DQAM, where $\omega$ is the degree of partial separability and $L'$ and $\bar{L}$ are the maximum and average of the block Lipschitz constants of the gradient of the quadratic penalty appearing in the augmented Lagrangian.
\end{abstract}

\section{Introduction}

With the rise and ubiquity of digital and data technology, practitioners in nearly all industries need to solve optimization problems of increasingly larger sizes. As a consequence,  new tools and methods are required to solve these big data problems, and to do so efficiently.

In this work, we are concerned with  convex optimization problems with an objective function that is separable into blocks of variables and where these blocks are linked by a subset of constraints which nevertheless make the problem nonseparable. Nonseparability is a source of difficulty in solving these very large optimization problems. This structure is particularly relevant in stochastic optimization problems where each block relates to a certain scenario and involves only variables related to that particular scenario. The objective function expressed as an expectation is separable in these blocks and the linking constraints (called non-anticipativity constraints)  encode the natural requirement that decisions be based only on information available at the time of decision making. Applications that can be modeled as large scale stochastic optimization problems include multicommodity network flow problems, financial planning problems and airline routing.

A classical approach to solving such problems is to use the augmented Lagrangian by relaxing the linking constraints. The augmented Lagrangian idea was first introduced independently by Hestenes \cite{hestenes69} and Powell \cite{powell72} and convergence of the associated augmented Lagrangian method was established later by Rockafellar \cite{rockafellar73, rockafellar76}. Advantages of this approach include the simplicity and stability of the multiplier iterations, the possibility of starting from an arbitrary multiplier, and the fact that there is no master problem to solve. However, the augmented Lagrangian is nonseparable, so the problem is still difficult to solve.

The nonseparability of the augmented Lagrangian has motivated the development of decomposition techniques. In an early work, Stephanopoulos and Westerberg \cite{stepwest75} suggest decomposing the augmented Lagrangian using linear approximations and Watanabe et al.\ \cite{watnismat78} use a transformation method to deal with the nonseparable cross products. The progressive hedging algorithm of Rockafellar and Wets \cite{Rockafellar91} also aims to tackle the nonseparability of the augmented Lagrangian. In a more recent line of work,  Ruszczy\'{n}ski \cite{Ruszczynski89, Ruszczynski95} and Mulvey and Ruszczy\'{n}ski \cite{Mulvey92, Mulvey95} propose and analyze a diagonal quadratic approximation (DQA) to the augmented Langrangian and an associated diaginal quadratic approximation method (DQAM). By approximating  the original problem by one that is separable into blocks, these techniques make a significant difference in terms of solvability because the problem is broken down into a number of problems of a more manageable size. Decomposition techniques have become even more attractive with the advances in parallel computing: since the decomposed subproblems can be solved independently,  parallelism is possible and this leads to acceleration.

A recent development in the area of decomposition techniques is the Expected Separable Overapproximation (ESO) of Richt\'arik and Tak\'{a}\v{c} and the associated parallel coordinate descent method (PCDM) presented in \cite{Richtarik12a} (this is discussed in detail in Section~\ref{S_ESOM}).

(Block) coordinate descent methods, early variants of which can be traced back to a 1870 paper of Schwarz \cite{Schwarz1870} and beyond, have recently become very popular due to their low per-iteration cost and good scalability properties. While convergence results were established several decades ago, iteration complexity bounds were not studied until recently \cite{Tseng01}. Randomized coordinate and block coordinate descent methods were proposed and analyzed in several settings, such as for smooth convex minimization problems \cite{Nesterov12, Richtarik12, RT:SPARS11a}, $L_1$-regularized problems \cite{ShalevTewari09},  composite problems \cite{Lu-Xiao2013,Richtarik12, TRG:Inexact2013}, nonsmooth convex problems \cite{Fercoq-Richtarik:SmoothedPCDM-2013}, nonconvex problems  \cite{Lu-Xiao2013b,Necoara13-nonconvex} and problems with separable constraints  \cite{Necoara12, Necoara_composite-coupled}. Parallel coordinate descent methods were developed and analyzed in  \cite{Bradley:PCD-paper, Richtarik12a, ChicagoICML13, Fercoq:Adaboost-2013, SSS2013-accelerated}, primal-dual methods in \cite{SSS2013, ChicagoICML13} and inexact methods in  \cite{TRG:Inexact2013}. The methods are used in a number of applications, including linear classification \cite{Lin:2008:DCDM,   Blondel2013, ChicagoICML13}, compressed sensing \cite{Li:CDOMACSGA}, truss topology design \cite{RT:TTD2011}, solving linear systems of equations \cite{2MIT2013} and group lasso problems \cite{Qin10}.





\subsection{Augmented Lagrangian} \label{S_AL}

Our work is motivated by the need to solve huge scale instances of constrained convex optimization problems of the form
\begin{subequations}
\label{Eq_problem}
\begin{eqnarray}
\label{Eq_original_objective}
     &\displaystyle\min_{x^{(1)},\dots, x^{(n)}}&  \sum_{i=1}^n g_i(\xbi)\\
\label{Eq_linking_constraints}
&\text{subject to }& \sum_{i=1}^n A_i\xbi =b\\
\label{Eq_separable_sets}
&\phantom{\text{subject to }}& \xbi \in X_i, \quad i = 1,2,\dots,n,
\end{eqnarray}
\end{subequations}
where for $i=1,2,\dots,n$ we assume that $X_i \subseteq \R^{N_i}$ are convex and closed sets, $g_i: \R^{N_i} \to \R \cup \{+\infty\}$ are convex and closed extended real-valued functions and $A_i \in \R^{m\times N_i}$.

While the objective function \eqref{Eq_original_objective} and the constraints \eqref{Eq_separable_sets} are separable in the decision vectors $x^{(1)},\dots,x^{(n)}$, the linear constraint \eqref{Eq_linking_constraints} links them together, which  makes the problem difficult to solve. Moreover, we are interested in the case when $n$ is very large (millions, billions and more), which introduces further computational challenges.

It will be useful to think of the decision vectors $\{x^{(i)}\}$ as ``blocks'' of a single decision vector $x \in \R^N$, with $N=\sum_i N_i$. This can be achieved as follows. We first partition the $N \times N$ identity matrix $I$ columnwise into $n$ submatrices $U_i \in \R^{N\times N_i}$, $i=1,2,\dots,n$, so that $I=[U_1,\dots,U_n]$, and then set $x = \sum_i U_i x^{(i)}$. That is, $x$ is the vector composed by stacking the vectors $x^{(i)}$ on top of each other.
It is easy to see that $x^{(i)} = U_i^T x \in \R^{N_i}$. Moreover, if we let \[A \eqdef \sum_{i=1}^n A_i U_i^T \in \R^{m\times N},\] then \eqref{Eq_linking_constraints} can be written compactly as $Ax = b$. Note also that\begin{equation}\label{Eq_0809s809s}A_i = A U_i, \qquad i=1,2,\dots,n.\end{equation} If we now write $g(x) \eqdef \sum_i g_i (x^{(i)})$ and $X \eqdef \sum_i U_i X_i \subseteq \R^N$, then  problem \eqref{Eq_original_objective}--\eqref{Eq_separable_sets} takes the following form:
\begin{subequations}
\label{Eq_problem2}
\begin{eqnarray}
\label{Eq_original_objective2}
     &\displaystyle\min_{x\in \R^N}&  g(x)\\
\label{Eq_linking_constraints2}
&\text{subject to }& Ax = b\\
\label{Eq_separable_sets2}
&\phantom{\text{subject to }}& x \in X.
\end{eqnarray}
\end{subequations}

A typical approach to overcoming the issue of nonseparability of  the linking constraint \eqref{Eq_linking_constraints2} is to drop it and instead consider the \emph{augmented Lagrangian},
\begin{eqnarray*}
     F_\pi(x) &\eqdef& g(x) + \ve{\pi}{b - Ax} + \tfrac{r}{2}  \|b - Ax\|^2,
\end{eqnarray*}
where $\pi\in \R^m$ is a vector of Lagrange multipliers, $r>0$ is a penalty parameter and $\|u\| = \ve{u}{u}^{1/2} = (\sum_{j} u_j^2)^{1/2}$ is the standard Euclidean norm. Now, the Method of Multipliers \cite{Bertsekas96,hestenes69} can be employed to solve problem \eqref{Eq_problem} as described below (Algorithm~\ref{A_MoM}).

\begin{algorithm}[H]
\caption{(Method of Multipliers)}\label{A_MoM}
  \begin{algorithmic}[1]
    \State \textbf{Initialization:} $\pi_0 \in \R^m$ and iteration counter $k=0$
    \While {the stopping condition has not been met}
    \State \textbf{Step 1:} Fix the multiplier $\pi_k$ and solve
\begin{subequations}
\begin{eqnarray}
     \displaystyle z_k \leftarrow \min_{x \in X} F_{\pi_k}(x).
     \label{Eq_Subproblem}
\end{eqnarray}
\State \textbf{Step 2:} Update the multiplier
\begin{equation}
     \pi_{k+1} \leftarrow \pi_k + r(b-Az_k),
\end{equation}
 \phantom{\textbf{Step 2:dfg}} and update the iteration counter $k \gets k+1$.
\end{subequations}
\EndWhile
  \end{algorithmic}
\end{algorithm}

\section{The Problem and Our Contributions}

The focus of this paper is on the optimization problem \eqref{Eq_Subproblem}. Hence, we need not be concerned about the dependence of $F$ on $\pi$ and will henceforth refer to the objective function, dropping the constant term $\ve{\pi}{b}$, as $F(x)$. Ignoring the constant term $\ve{\pi}{b}$, problem \eqref{Eq_Subproblem} is a \emph{convex composite }optimization problem, i.e., a problem of the form
\begin{equation}
\label{E_FaugLag}
     \min_{x\in \R^N}  \{F(x)\eqdef f(x)+\Psi(x)\},
\end{equation}
where $f$ is a smooth convex function and $\Psi$ is a separable (possibly nonsmooth) convex function. Indeed, we may set

\begin{equation}
\label{D_f}
  f(x)\eqdef \frac{r}{2}\|b-Ax\|^2 =  \frac{r}{2}  \left\|b-\sum_{i=1}^nA_i\xbi \right\|^2,
\end{equation}
and
\[
\Psi(x) \eqdef \begin{cases} g(x) - \ve{\pi}{Ax}, & x \in X,\\
+\infty, & \text{otherwise.}
\end{cases}
\]

The main purpose of this work is to draw links between two existing decomposition methods for solving \eqref{E_FaugLag}, one old and one new,  both based on separable approximations to the objective function. In particular, we consider DQAM  of Mulvey and Ruszczy\'nski \cite{Mulvey92,Mulvey95,Ruszczynski95} and PCDM of Richt\'{a}rik and Tak\'{a}\v{c} \cite{Richtarik12a}, respectively. Our main contributions (not in order of significance) include:


\begin{enumerate}

\item \textbf{Two measures of separability.} We show that the parameter ``number of neighbours'', used in the analysis of DQAM \cite{Ruszczynski95}, and the degree of partial separability, used in the analysis of PCDM \cite{Richtarik12a}, coincide up to an additive constant in the case of quadratic $f$.

\item \textbf{Two generalizations of DQAM.} We provide a simplified derivation of the diagonal quadratic approximation, which enables us to propose two generalizations of DQAM (Section~\ref{S_2gen}) to non-quadratic functions $f$, based on
    \begin{enumerate}
    \item[(i)] a finite difference separable approximation to the augmented Lagrangian (Algorithm~\ref{A_DQA-FD}), and
    \item[(ii)]a quadratic approximation with the Hessian matrix replaced by an approximation of its block diagonal (Algorithm~\ref{A_DQA-Hessian}).
    \end{enumerate}
    We do not study the complexity of these algorithms in this paper.

\item \textbf{Equivalence of PCDM and DQAM for smooth problems.} We identify a situation in which the second of our generalizations of DQAM (Algorithm~\ref{A_DQA-Hessian}) coincides with a ``fully parallel'' variant of PCDM (Algorithm~\ref{A_PCD_DSO}) for an appropriate selection of parameters of the method (see Section~\ref{S_97d070dsds}, Theorem~\ref{T_coincide}). This happens for problems with arbitrary smooth $f$ and $\Psi\equiv 0$. 

\item \textbf{Improved complexity of PCDM under strong convexity.} We derive an improved complexity result for PCDM in the case when $F$ is strongly convex (Section~\ref{S_Complexity}, Theorem~\ref{Thm_complexity}). The result is much better than that in \cite{Richtarik12a} in situations where the strong convexity constant of $F$ is much larger than the sum of the strong convexity constants of the constituent functions $f$ and $\Psi$.

\item  \textbf{Versatility of PCDM.}  PCDM enjoys complexity guarantees even in the case when $F$ is merely convex, as opposed to it being strongly convex. Moreover, PCDM is flexible in that it allows for an \emph{arbitrary} number of block updates per iteration, whereas DQAM needs to update \emph{all blocks}.

\item \textbf{Complexity in the strongly convex case.}
      We study the newly developed complexity guarantees for (fully parallel variant of) PCDM (Algorithm~\ref{A_PCD_DSO}) and the existing convergence rates for DQAM and show that even though DQAM is specifically designed to approximate the augmented Lagrangian, PCDM has much better theoretical guarantees (Section~\ref{S_complexity_comp}). In particular, if $F$ is strongly convex, both DQAM and  PCDM converge linearly;    that is,  $F(x_{k+1}) \leq q F(x_k)$, where $q$ depends on the method. However, we show that  $q$ is much better (i.e., smaller) for PCDM than for DQAM, which then leads to vast speedups in terms of iteration complexity. In particular, we show that the theoretical bound for the number of iterations required to find an $\epsilon$-approximate solution is at least \begin{equation}\label{eq:090909oijsod}\frac{16(\omega-1)^3}{\omega}\times \frac{L'}{\bar{L}} \qquad (\geq 8\tfrac{L'}{\bar{L}}(\omega-1)^2 \text{ for } \omega\geq 2)\end{equation}
    \emph{times larger} for DQAM than for (fully parallel) PCDM. Here, $\omega$ is the degree of partial separability\footnote{The multiplicative improvement factor \eqref{eq:090909oijsod} is only valid for $\omega \geq 2$ as DQAM was not analyzed in the case $\omega=1$.} of $f$ (defined in Section~\ref{S_SEP}), and $L'$ and $\bar{L}$ are the maximum and average of the constants $L_i = r\|A_i^T A_i\|$,  $i=1,2,\dots,n$, respectively. Note that the speedup factor  \eqref{eq:090909oijsod} is larger than 1000 for $\omega=10$ even in the case when $L'=\bar{L}$. In practice, however, $L'$ will typically be larger than $\bar{L}$, often much larger.

    The form of the speedup factor \eqref{eq:090909oijsod} comes from the fact that DQAM depends on $(\omega-1)^3$ and $L'$ (while PCDM depends on $\omega$ and $\bar{L}$), which adversely affects its theoretical complexity rate. Let us comment that Mulvey and Ruszczy\'{n}ski \cite{Mulvey92} remarked that the dependence of DQAM on $\omega$  is  in practice much better than cubic, although this was not previously established theoretically. We thus answer their conjecture in the affirmative, albeit for a (as we shall see, not so very) different method. To the best of our knowledge,  no improved results were available in the literature up to this point.

\item \textbf{Optimal number of block updates per iteration.} We show that under a simple parallel computing model it is optimal for PCDM to update as many block in a single iteration as there are parallel processors (Section~\ref{S_OPT}, Theorem~\ref{T_opt}). As a consequence, the DQAM approach of updating \emph{all} blocks in a single iteration is less than optimal.

\item \textbf{Computations.}    We also provide preliminary numerical results that show the practical advantages of PCDM.
\end{enumerate}

\section{Two Measures of Separability} \label{S_SEP}

In this section we provide a link between the measures of separability of $f$ utilized in the analysis of DQAM \cite{Mulvey92} and PCDM \cite{Richtarik12a}.
In the first case, the quantity is defined specifically for a quadratic objective; in the second case the definition is general. As we shall see, both quantities coincide in the quadratic case. As the complexity of the two methods depends on these quantities, our observation allows us to compare the convergence rates.
Both measures of separability are to be understood with respect to the fixed block structure introduced before.

We first define a separability measure introduced for the convex quadratic $f(x)=\tfrac{r}{2}\|b-Ax\|^2$ by Ruszczy\'{n}ski \cite[Section 3]{Ruszczynski95} (and called the ``number of neighbors'' therein).

Let $A_{ji}$ be the $j$-th row of matrix $A_i$. Let $m_i$ be the number of nonzero rows in $A_i$ and for each $i$  define  an $m\times m_i$ matrix $E^{i}$ as follows: $E^i_{jl} =1$  if $A_{ji}$ is the $l$-th consecutive nonzero row of the matrix $A_i$, and 0 otherwise. Note that  $E^i$ is a matrix containing zeros and $m_i$ ones, one in each  column. Further, for any $i\in \{1,2,\dots,n\}$ and $u\in \{1,2,\dots,m_i\}$
define \begin{equation}\label{eq:090909gfgf}V(i,u)\eqdef \{(i',u')\;:\; i\in\{1,2,\dots,n\}, \; u'\in \{1,2,\dots,m_{i'}\},\; k\neq i,\; \langle E^{i}_u, E^{i'}_{u'} \rangle\neq 0\},\end{equation}
where $E^i_u$ is the $u$-th column of matrix $E^i$.

\begin{definition}[Ruszczy\'{n}ski separability] The Ruszczy\'{n}ski degree of separability of the function $f$ defined in \eqref{D_f}  is
\begin{equation}\label{eq:omega_R}\omega_R = \max \{|V(i,u)| \;:\; i=1,2,\dots,n, \; u=1,2,\dots,m_i\}.\end{equation}
\end{definition}

We now define the measure of separability used by Richt\'{a}rik and Tak\'{a}\v{c} \cite{Richtarik12a} in the analysis of PCDM.

\begin{definition}[Partial separability]

A smooth convex function $f:\R^N \to \R$ is  partially separable of degree $\omega$ if there exists a collection ${\cal J}$ of subsets of $\{1,2,\ldots, n\}$ such that
\begin{equation}
\label{E_separability}
     f(x)=\sum_{J\in \mathcal{J}}f_{J}(x) \qquad \text{ and }    \qquad \max_{J\in \mathcal{J}}|J| \leq \omega,
\end{equation}
where for each $J$,  $f_{J}$ is a smooth convex function that depends on $x^{(i)}$ for $i\in J$ only.
\end{definition}


Our first result says that  in the case of convex quadratics, the two measures of separability defined above coincide. This will allow us to provide a direct comparison of the complexity results of PCDM and DQAM.

\begin{theorem}
For convex quadratic function $f$ given by \eqref{D_f} we have $\omega = \omega_R+1$.
\end{theorem}
\begin{proof}
First, we can write
\begin{equation}
f(x) = \frac{r}{2}\sum_{j=1}^m\left(b_j - \sum_{i=1}^n A_{ji}\xbi\right)^2,
\label{Eqn_Phir}
\end{equation}
where $b_j$ is the $j$-th entry of $b$. Note that all summands in the decomposition are convex and smooth. Moreover, summand $j$ depends on  $x^{(i)}$ if and only if $A_{ji}\neq 0$. If we now let
\begin{equation}\label{eq:omega_j} \omega_j = |\{i \;:\; A_{ji}\neq 0\}|, \qquad j=1,2,\dots,m,
\end{equation} then we conclude  that $f$ is partially separable of degree \begin{equation}\label{eq:bb564}\omega = \max_{j\in \{1,2,\dots,m\}} \omega_j.\end{equation}
In the rest of the proof we  proceed in two steps.
\begin{enumerate}
\item[(i)] Let us fix $i\in \{1,2,\dots,n\}$, $u\in \{1,2,\dots,m_i\}$ and let $j=j(i,u)$ be such row index for which $E^i_{ju} = 1$.
Note that, since $E^i$ is a 0-1 matrix with exactly one entry of each column equal to 1, we have $E^i_{j'u} =0$ for all $j'\neq j$. This means that for any $i' \in\{1,2,\dots,n\}$ and $u'\in\{1,2,\dots,m_{i'}\}$,
\begin{equation}\label{eq:98udd8}\ve{E^i_u}{E^{i'}_{u'}} \neq 0 \quad \Leftrightarrow \quad E^{i'}_{ju'} =1.\end{equation}
Likewise, $E^{i'}$ has at most entry equal to 1 in each row. Moreover, the $j$-th row of $E^{i'}$
contains $1$ precisely when $A_{ji'}\neq 0$. This means that
\begin{equation}\label{eq:09dsjsdisd4343} |\{u' \;:\; E^{i'}_{ju'} = 1\} | =
\begin{cases} 1 \quad & \text{if} \quad A_{ji'} \neq 0,\\
0 \quad & \text{if} \quad A_{ji'}= 0.
\end{cases}
\end{equation}
We now have
\begin{eqnarray}
|V(i,u)|&\overset{\eqref{eq:090909gfgf}}{=}& |\{(i',u') \;:\; i'\neq i,\; \ve{E^i_u}{E^{i'}_{u'}}\neq 0 \}|\notag\\
&\overset{\eqref{eq:98udd8}}{=}& |\{(i',u') \;:\; i'\neq i,\; E^{i'}_{ju'}=1 \}|\notag\\
&=& \sum_{i' \neq i} |\{u'\;:\; E^{i'}_{ju'}=1\}| \notag\\
&\overset{\eqref{eq:09dsjsdisd4343}+\eqref{eq:omega_j}}{=}& \omega_j - 1.\label{eq:lklk0939}
\end{eqnarray}
\item[(ii)] Building on the result from part (i), we can now write
\begin{eqnarray*}
\omega_R &\overset{\eqref{eq:omega_R}}{=} & \max \{|V(i,u)|\;:\;i\in\{1,2,\dots,n\},\; u\in\{1,2,\dots,m_i\}\}\\
&\overset{\eqref{eq:lklk0939}}{=}& \max \{\omega_{j(i,u)} - 1\;:\; i\in\{1,2,\dots,n\},\; u\in\{1,2,\dots,m_i\}\\
&=& \max_{j\in \{1,2,\dots,m\}} \omega_j - 1\\
&\overset{\eqref{eq:bb564}}{=}& \omega - 1.
\end{eqnarray*}
In the third identity above we used the simple observation that every row $j \in\{1,2,\dots,m\}$ for which $\omega_j\neq 0$ can be written as $j=j(i,u)$ for any $i$ for which $A_{ji}\neq 0$, and some $u$ (which depends on $i$).
\end{enumerate}

%
%
\end{proof}

Let us remark that besides \eqref{Eqn_Phir}, we could have decomposed $f$ also as
\begin{equation}
f(x)= \frac{r}{2}\left(\|b\|^2 - 2\sum_{i=1}^n\langle b,A_i\xbi\rangle+\sum_{i=1}^n\sum_{j=1}^n \langle A_i\xbi,A_j x^{(j)}\rangle\right),
\label{Eq_separated1}
\end{equation}
with each summand depending on at most 2 blocks of $x$. However, we \emph{cannot} conclude that $f$ is partially separable of degree 2 because the terms are \emph{not} all convex, which is required in the definition of partial separability.


\section{Diagonal Quadratic Approximation Method}\label{S_DQA}

In this section we present the Diagonal Quadratic Approximation Method (DQAM) that was introduced and analysed in a series of papers by Mulvey and Ruszczy\'{n}ski \cite{Mulvey92,Mulvey95}, Ruszczy\'{n}ski \cite{Ruszczynski95} and Berger, Mulvey and Ruszczy\'{n}ski \cite{Berger94}. As explained in Section~\ref{S_AL}, the augmented Lagrangian is nonseparable because of the cross products $\langle A_ih^{(i)},A_j h^{(j)} \rangle$ appearing in $f(x+h)$. The DQAM provides a separable approximation of $f(x+h)$ by ignoring these cross terms; this approximation is referred to as the diagonal quadratic approximation (DQA). This makes Step 1 of the method of multipliers (\eqref{Eq_Subproblem} in Algorithm~\ref{A_MoM}) significantly easier to solve, and amenable to parallel processing.

First, notice that we can write
\begin{eqnarray}
\notag f(x+h)
\notag &=& \tfrac{r}{2}\|b-A(x+h)\|^2 \\ 
\notag &=& \tfrac{r}{2}\|b\|^2 - r\ve{b}{A(x+h)} + \tfrac{r}{2}\left(\|Ax\|^2 + 2\ve{Ax}{Ah} + \|Ah\|^2\right)\\
\notag &=& f(x) + \ve{f'(x)}{h} + \tfrac{r}{2}\|Ah\|^2\\
\notag &=& f(x) + \ve{ f'(x)}{h} + \tfrac{r}{2}\sum_{i=1}^n \|A_i h^{(i)}\|^2 + \tfrac{r}{2}(\|Ah\|^2 - \sum_{i=1}^n\|A_i h^{(i)}\|^2)\\
\label{E_DQA_f_formulation}
  &=& f(x) + \sum_{i=1}^n\ve{( f'(x))^{(i)}}{h^{(i)}} + \tfrac{r}{2}\sum_{i=1}^n \|A_i h^{(i)}\|^2 + \tfrac{r}{2}\sum_{i\neq j} \ve{A_i h^{(i)}}{A_j h^{(j)}}.
\end{eqnarray}

Now observe that it is only the last term in \eqref{E_DQA_f_formulation}, composed of products $\ve{A_i h^{(i)}}{A_j h^{(j)}}$ for $i\neq j$, which is not separable. Ignoring these terms, we get a separable approximation of $f(x+h)$ in $h$,
\begin{eqnarray}
\label{E_fDQA}
f(x+h) \approx f^{\text{DQA}}(x+h) &\eqdef& f(x) + \ve{ f'(x)}{h} + \frac{r}{2}\sum_{i=1}^n \|A_i h^{(i)}\|^2,
\end{eqnarray}
which in turn leads to a separable approximation of $F(x+h)$ in $h$:
\begin{equation}
\label{E_FDQA}
  F(x+h) \overset{\eqref{E_FaugLag}}{=} f(x+h) + \Psi(x+h)  \overset{\eqref{E_fDQA}}{\approx} f^{\text{DQA}}(x+h) + \Psi(x+h).
\end{equation}

Mulvey and Ruszczy\'{n}ski \cite{Mulvey92} propose a slightly less transparent construction of the same approximation. For a fixed  $x$, they approximate $f(y)$ via replacing the cross-products $\ve{A_i y^{(i)}}{A_j y^{(j)}}$, for $i\neq j$, by
\begin{equation}\label{eq:09jsdjsdsd}\ve{A_i y^{(i)}}{A_j x^{(j)}} + \ve{A_i x^{(i)}}{A_j y^{(j)}} - \ve{A_i x^{(i)}}{A_j x^{(j)}}.\end{equation}
Clearly, this is equivalent to what we do above, which can be verified by substituting $y=x+h$ into \eqref{eq:09jsdjsdsd}.

\subsection{The algorithm}

We now present the DQA method (Algorithm~\ref{A_DQA}). The algorithm replaces Step 1 of the Method of Multipliers (Algorithm \ref{A_MoM}). In what follows, $\theta\in (0,1)$ is a user defined parameter.

\begin{algorithm}[H]
\caption{(DQAM: Diagonal Quadratic Approximation Method)}\label{A_DQA}
  \begin{algorithmic}[1]
    \For {$k = 0,1,2,\dots$}
    \State \textbf{Step 1a:} Solve for $h_k$
    \begin{subequations}
    \begin{equation}
\label{DQA_step1a}
     h_k \gets  \arg \min_{h \in \R^N} \left\{ f^{\text{DQA}}(x_k+h) + \Psi(x_k + h)\right\}
     \end{equation}
\State \textbf{Step 1b:} Determine intermediate vector $y_k$
\begin{equation}
\label{DQA_step1b}
     y_k \gets  x_k + h_k
\end{equation}
\State \textbf{Step 1c:} Form the new iterate $x_{k+1}$
\begin{equation}
     \label{DQA_step1c}
          x_{k+1} \gets (1-\theta)x_k + \theta y_k
\end{equation}
\end{subequations}
\EndFor
  \end{algorithmic}
\end{algorithm}

Let us now comment on the individual steps of Algorithm~\ref{A_DQA}. Step 1a is easy to  execute because the function  that is being minimized in \eqref{DQA_step1a} is separable in $h$, and hence the problem decomposes into $n$ independent \emph{lower-dimensional} problems:
\[h_k^{(i)} = \arg \min_{h^{(i)} \in \R^{N_i}} \left\{ \ve{(f'(x_k))^{(i)}}{h^{(i)}} + \frac{r}{2} \|A_i h^{(i)}\|^2  + \Psi_i(x_k^{(i)} + h^{(i)})\right\}, \quad i=1,2,\dots,n.\]
Moreover, the problems are \emph{independent}, and hence the updates $h_k^{(1)}, \dotsm h_k^{(n)}$ can be computed \emph{in parallel.}
In \eqref{DQA_step1b} an intermediate vector $y_k$ is formed, and then in \eqref{DQA_step1c} a convex combination of the current iterate $x_k$ and the intermediate vector $y_k$ is taken to produce the new iterate $x_{k+1}$. Step \eqref{DQA_step1c} is needed because DQAM uses a local approximation, so if the new point $x_k+h_k$ is far from $x_k$, the approximation error may be too big and a reduction in the objective function value is not guaranteed. This would lead to serious stability and convergence problems in general, and hence, Step 1c is employed as a correction step for regularizing the method.

\subsection{Two generalizations} \label{S_2gen}

DQAM was originally designed and analyzed for convex quadratics. Here we propose two generalizations of the method to non-quadratic convex functions $f$. Our generalizations are based on the following simple result.


\begin{proposition} If $f(x) = \tfrac{r}{2}\|b-Ax\|^2$, then for all $x,h\in \R^N$,
\label{L_DQA_fapprox}
\begin{equation}
\label{E_approx}
  f^{\text{DQA}}(x+h) = f(x) + \sum_{i=1}^n \left[f(x+U_i h^{(i)}) - f(x)\right]
\end{equation}
and
\begin{eqnarray}
\label{E_approx2}
  f^{\text{DQA}}(x+h) &=& 
     f(x) + \sum_{i=1}^n \left[ \ve{(f'(x))^{(i)}}{h^{(i)}} +  \tfrac{1}{2} \ve{C_i(x) h^{(i)}}{h^{(i)}}\right],
\end{eqnarray}
where $C_i(x) = U_i^T f''(x) U_i$.


\end{proposition}
\begin{proof}First note that
\begin{eqnarray*} \sum_{i=1}^n \left[f(x+U_i h^{(i)}) - f(x)\right]  
&=& \sum_{i=1}^n \left[\tfrac{r}{2}\|b-Ax - A_i h^{(i)}\|^2 - \tfrac{r}{2}\|b-Ax\|^2\right]\\
&=& \sum_{i=1}^n \left[r \ve{Ax-b}{A_i h^{(i)}} + \tfrac{r}{2}\|A_ih^{(i)}\|^2\right]\\
&=& \ve{f'(x)}{h} + \tfrac{r}{2}\sum_{i=1}^n \|A_i h^{(i)}\|^2,
\end{eqnarray*}
which, in view of \eqref{E_fDQA}, establishes \eqref{E_approx}. Finally, \eqref{E_approx2} follows from \eqref{E_fDQA} and the fact that
\[\tfrac{r}{2}\|A_i h^{(i)}\|^2 = \tfrac{1}{2}\ve{U_i^T f''(x) U_i h^{(i)}}{h^{(i)}},\]
which in turn follows from the identities $f''(x) = rA^TA$ and $A_i = AU_i$.
\end{proof}

Our two generalized methods are obtained by replacing $f^{\text{DQA}}(x+h)$ in Step 1 of Algorithm~\ref{A_DQA} by one of the two approximations \eqref{E_approx} and \eqref{E_approx2} (in the second case we allow for $C_i(x)$ to be an arbitrary positive semidefinite matrix and not necessarily $U_i^T f''(x) U_i$), leading to Algorithm~\ref{A_DQA-FD} and Algorithm~\ref{A_DQA-Hessian}, respectively.

\begin{algorithm}[H]
\caption{(Generalization of DQAM: Finite Differences Approximation)}\label{A_DQA-FD}
  \begin{algorithmic}[1]
    \For {$k = 0,1,2,\dots$}
    \State \textbf{Step 1a:} Solve for $h_k$
    \begin{subequations}
    \begin{equation}
\label{DQAFD_step1a}
     h_k \gets  \arg \min_{h \in \R^N} \left\{ f(x_k) + \sum_{i=1}^n \left[f(x_k + U_i h^{(i)}) - f(x_k)\right] + \Psi(x_k + h)\right\}
     \end{equation}
\State \textbf{Step 1b:} Determine intermediate vector $y_k$
\begin{equation}
\label{DQAFD_step1b}
     y_k \gets  x_k + h_k
\end{equation}
\State \textbf{Step 1c:} Form the new iterate $x_{k+1}$
\begin{equation}
     \label{DQAFD_step1c}
          x_{k+1} \gets (1-\theta)x_k + \theta y_k
\end{equation}
\end{subequations}
\EndFor
  \end{algorithmic}
\end{algorithm}

Algorithm~\ref{A_DQA-FD} is based on a finite difference approximation, and is applicable to (possibly) nonsmooth functions. Algorithm~\ref{A_DQA-Hessian} is based on a separable quadratic approximation. To the best of our knowledge, these algorithms have not been previously proposed, with the exception of the case when $f$ is a convex quadratic when both methods coincide with DQAM.

\begin{algorithm}[H]
\caption{(Generalization of DQAM: Separable Quadratic Approximation)}\label{A_DQA-Hessian}
  \begin{algorithmic}[1]
    \For {$k = 0,1,2,\dots$}
    \State \textbf{Step 1a:} Solve for $h_k$
    \begin{subequations}
    \begin{equation}
\label{DQAHessian_step1a}
     h_k \gets  \arg \min_{h \in \R^N} \left\{ f(x_k) + \ve{f'(x_k)}{h} + \tfrac{1}{2}\sum_{i=1}^n \ve{C_i(x_k) h^{(i)}}{h^{(i)}} + \Psi(x_k + h)\right\}
     \end{equation}
\State \textbf{Step 1b:} Determine intermediate vector $y_k$
\begin{equation}
\label{DQAHessian_step1b}
     y_k \gets  x_k + h_k
\end{equation}
\State \textbf{Step 1c:} Form the new iterate $x_{k+1}$
\begin{equation}
     \label{DQAHessian_step1c}
          x_{k+1} \gets (1-\theta)x_k + \theta y_k
\end{equation}
\end{subequations}
\EndFor
  \end{algorithmic}
\end{algorithm}

In this paper we do not analyze any of these methods. Instead, we propose that DQAM be replaced by PCDM, described in the next section.


\section{Parallel Coordinate Descent Method}\label{S_ESOM}

As discussed in the introduction,  we propose that instead of implementing Step 1 of the Method of Multipliers (Algorithm \ref{A_MoM}) using DQAM, a parallel coordinate descent method (PCDM) be used instead. This section is devoted to describing the method, developed by Richt\'{a}rik and Tak\'{a}\v{c} \cite{Richtarik12a}.

\subsection{Block samplings}

As we shall see, unlike DQAM where all blocks are updated at each iteration, PCDM allows for an (almost) arbitrary random subset of blocks to be updated
at each iteration. The purpose of this section is to formalize this.

In particular, at iteration $k$ only blocks $i \in S_k \subseteq \{1,2,\dots,n\}$ are updated, where $\{S_k\}$, $k\geq 0$, are iid random sets having the following two properties:
\begin{equation}\label{eq:uddd8}\Prob(i \in S_k) = \Prob(j \in S_k) \qquad \text{for all} \qquad i,j \in \{1,2,\dots,n\},\end{equation}
\begin{equation}\label{eq:uddd9}\Prob(i \in S_k) > 0  \qquad \text{for all} \qquad i \in \{1,2,\dots,n\}.\end{equation}
It is easy to see that, necessarily, $\Prob(i \in S_k) = \tfrac{\Exp [|S_k|]}{n}$. Following \cite{Richtarik12a}, for simplicity we refer to an arbitrary random set-valued mapping with values in the power set $2^{\{1,2,\dots,n\}}$ by the name \emph{block sampling}, or simply \emph{sampling}. A sampling $S_k$ is called \emph{uniform} if it satisfies \eqref{eq:uddd8} and \emph{proper} if it satisfies \eqref{eq:uddd9}.

In \cite{Richtarik12a}, PCDM was analyzed for all proper uniform samplings. However, better complexity results were obtained for so called  \emph{doubly uniform} samplings, which belong to the family of uniform samplings. For brevity purposes, in this paper we concentrate on a subclass of doubly uniform samplings called \emph{$\tau$-nice} samplings, which we now define.

\begin{definition}[$\tau$-nice sampling]
Let $\tau$ be an integer between 1 and $n$. A sampling $\hat{S}$ is called $\tau$-nice if for all $S \subseteq \{1,2,\dots,n\}$,
\[\Prob(\hat{S} = S) = \begin{cases} 0, &  \quad  |S|\neq \tau, \\ \tfrac{1}{{n \choose \tau }}, & \quad \text{otherwise.}\end{cases}\]
\end{definition}

A natural candidate for $\tau$ is the number of available processors/threads as then updates to the $\tau$ blocks of $x_k$ can be computed in parallel. As we shall later see, this is also the optimal choice from the complexity point of view (Theorem~\ref{T_opt}).

\subsection{Expected Separable Overapproximation (ESO)} \label{S_ESO}

Fixing positive scalars $w_1,\dots,w_n$ (we write $w = (w_1,\dots,w_n)$), let us define a separable norm on $\R^N$ by
\begin{equation}
\label{S_Norms_1}
     \|x\|_w \eqdef  \left(\sum_{i=1}^n w_i \|x^{(i)}\|^2_{(i)}\right)^{1/2}, \quad x \in \R^N,
\end{equation}
where for each $i=1,2,\dots,n$ we fix a positive definite matrix $B_i \in \R^{N_i \times N_i}$ and set
\begin{eqnarray}
\label{S2_Norm_def}
  \|t\|_{(i)} \eqdef \langle B_i t,t \rangle^{1/2}, \quad t\in \R^{N_i}.
\end{eqnarray}

We can now define the concept of expected separable overapproximation.

\begin{definition}[Expected Separable Overapproximation (ESO) \cite{Richtarik12a}]
\label{D_ESO}
     Let $\beta $ and $w_1,\dots,w_n$ be positive constants and $\hat{S}$ be a proper uniform sampling. We say that $f:\R^N \to \R$ admits a $(\beta,w)$-ESO with respect to $\hat{S}$  (and, for simplicity, we write $(f,\hat{S}) \sim ESO (\beta,w)$) if for all $x,h \in \R^N$,
\begin{equation}
\label{E_ESO_f}
     \E\left[f\left(x+ {\textstyle \sum_{i\in \hat{S}} U_i h^{(i)} } \right)\right] \leq  f(x) + \frac{\E[|\hat{S}|]}{n} \left(\langle f'(x),h \rangle + \frac{\beta}{2} \|h\|_{w}^2\right).
\end{equation}

\end{definition}

In Section~\ref{S_algPCDM} we describe how the ESO is used to design a parallel coordinate descent method for solving problem \eqref{E_FaugLag}.
The issue of how the parameters $w$ and $\beta$ giving rise to an  ESO can be determined/computed will be discussed in Section~\ref{S_ESO_partially_sep}.

\subsection{The algorithm} \label{S_algPCDM}

Unlike with DQAM, were $f$ is replaced by $f^{\text{DQA}}$ and $\Psi$ is kept intact, PCDM replaces \emph{both} $f$ and $\Psi$. This is because in PCDM we compute an approximation to
\begin{equation}\label{eq:09809ds}\E \left[ F(x + {\textstyle \sum_{i \in \hat{S}}} U_i h^{(i)})\right] = \E\left[ f(x + {\textstyle \sum_{i \in \hat{S}}} U_i h^{(i)}) + \Psi(x + {\textstyle \sum_{i \in \hat{S}}} U_i h^{(i)}) \right],\end{equation}
which (unless $|\hat{S}|=n$) affects $\Psi$ as well. It can be verified (see \cite[Section~3]{Richtarik12a}) that due to separability of $\Psi$ the following identity holds:
\begin{equation}\label{eq:d09uddsj}\Exp\left[\Psi(x + {\textstyle \sum_{i \in \hat{S}}} U_i h^{(i)}) \right] = \left(1-\frac{\Exp[|\hat{S}|]}{n}\right) \Psi(x) +\frac{\Exp[|\hat{S}|]}{n} \Psi(x+h).\end{equation}
Substituting \eqref{eq:d09uddsj} and \eqref{E_ESO_f} into \eqref{eq:09809ds}, we obtain
\begin{equation}\label{eq:09s0asjdxxXX} \E \left[ F(x + {\textstyle \sum_{i \in \hat{S}}} U_i h^{(i)})\right] \leq F^{\text{ESO}}(x+h) \eqdef \left(1-\frac{\Exp[|\hat{S}|]}{n}\right) F(x) + \frac{\Exp[|\hat{S}|]}{n} H_{\beta,w}(x+h),\end{equation}
where
\begin{equation}
\label{H}
     H_{\beta,w}(x+h) \eqdef f(x) + \langle f'(x) ,h \rangle + \frac{\beta}{2} \|h\|_w^2 + \Psi(x+h),
\end{equation}
which is separable in $h$:
\begin{equation}
\label{H2}
     H_{\beta,w}(x+h) \overset{\eqref{S_Norms_1}+\eqref{S2_Norm_def}}{=} f(x) + \sum_{i=1}^n\left\{\langle (f'(x))^{(i)} ,h^{(i)} \rangle + \frac{\beta w_i}{2} \ve{B_i h^{(i)}}{h^{(i)}} + \Psi_i(\xbi+h^{(i)})\right\}.
\end{equation}

We are now ready to present the parallel coordinate descent method (Algorithm~\ref{A_PCD_ESO}).

\begin{algorithm}[H]
\caption{(PCDM: Parallel Coordinate Descent Method)}\label{A_PCD_ESO}
  \begin{algorithmic}[1]
  \State \textbf{Initialization:} $x_0\in \R^N$,  ESO parameters $(\beta,w)$
    \For {$k = 0,1,2,\dots$}
    \State \textbf{Step 1a:} Solve
    \begin{subequations}
\begin{equation}
\label{PCDM_step1}
     h_k \gets \arg \min_{h \in \R^{N}} F^{\text{ESO}}(x_k+h)
\end{equation}
\State \textbf{Step 1b:} Update $x_k$
\begin{equation}
\label{PCDM_step1b}
     x_{k+1} \gets x_k + \sum_{i \in S_k} U_i h_k^{(i)}
\end{equation}
\end{subequations}
\EndFor
  \end{algorithmic}
\end{algorithm}

Given an iterate $x_k$, in \eqref{PCDM_step1} we compute \begin{equation}\label{E_hi} h_k = h(x_k) \eqdef \arg \min_{h \in \R^N} F^{\text{ESO}}(x_k+h) \overset{\eqref{eq:09s0asjdxxXX}}{=} \arg \min_{h \in \R^N} H_{\beta,w}(x_k + h).\end{equation}
Further, note that  \eqref{PCDM_step1b} is equivalent to writing
\[x_{k+1}^{(i)} = \begin{cases}x_k^{(i)}, & \quad i \notin S_k, \\
x_k^{(i)} + h_k^{(i)}, & \quad i \in S_k.\end{cases}\]
That is, only blocks belonging to the random set $S_k$ are updated. This means that in \eqref{PCDM_step1} we need not compute all blocks of $h_k$. In view of \eqref{H2} and \eqref{E_hi},  this is possible, and hence \eqref{PCDM_step1} can be replaced by
\begin{equation}\label{PCDM_step1-separable}
     h_k\ii \gets \arg \min_{h^{(i)} \in \R^{N_i}} \left\{\langle (f'(x_k))^{(i)} ,h^{(i)}\rangle + \frac{\beta w_i}{2}\ve{B_i h^{(i)}}{h^{(i)}} + \Psi_i(x_k^{(i) } + h^{(i)})\right\}, \quad i \in S_k.
\end{equation}

\subsection{ESO for partially separable smooth convex functions}\label{S_ESO_partially_sep}

In order for PCDM to be implementable, one needs first to compute the parameters $w_1, \dots, w_n$ (defining the norm $\|\cdot\|_w$) and $\beta>0$ for which $(f,\hat{S})\sim ESO(\beta,w)$, i.e., for which \eqref{E_ESO_f} holds. Clearly, the parameters $\beta$ and $w$ depend on $f$ and $\hat{S}$.

In what follows we will assume that the gradient of $f$ is block Lipschitz. That is, there exist positive constants $L_1,\dots, L_n$ such that for all $x\in \R^N$, $i \in \{1,2,\dots,n\}$ and $h^{(i)}\in \R^{N_i}$,
\begin{equation}
\label{S2_Lipschitz}
     \| (f'(x + U_i t))^{(i)} - ( f'(x))^{(i)} \|_{(i)}^* \leq L_i \|t\|_{(i)},
\end{equation}
where $\|s\|_{(i)}^* \eqdef \max \{\ve{s}{x} \;:\; \|x\|_w=1\} = \ve{B_i^{-1}s}{s}^{1/2}$ is the conjugate norm to $\|\cdot\|_w$.

\begin{theorem}[Theorem 14 in \cite{Richtarik12a}] \label{T_ESO_omega}
Assume $f$ is convex, partially separable of degree $\omega$, and has block Lipschitz gradient with constants $L_1,L_2,\dots,L_n>0$. Further, assume that $\hat{S}$ is a $\tau$-nice sampling, where $\tau \in \{1,2,\dots,n\}$. Then \[(f,\hat{S})\sim ESO(\beta,w),\] where
\begin{equation}\label{eq:beta09809809}\beta = 1+\frac{(\omega-1)(\tau-1)}{\max\{1,n-1\}}, \qquad w_i = L_i, \qquad i=1,2,\dots,n.\end{equation}
\end{theorem}

In Section~\ref{S_Complexity} we study the complexity of PCDM in the case covered by the above theorem (and under a further strong convexity assumption).

Consider now the special case of convex quadratic $f$ given by \eqref{D_f}. If the matrices $A_i^TA_i$, $i=1,2,\dots,n$, are all positive definite, we can choose $B_i=rA_i^T A_i$, $i=1,2,\dots,n$, in which case we will have $L_i=1$ for all $i$. Otherwise we  can  choose $B_i$ to be the $N_i\times N_i$ identity matrix, and  then \begin{equation}\label{eq:08998as8s}L_i = r\|A_i^TA_i\| \eqdef r \max_{\|h^{(i)}\|\leq 1} \|A_i^T A_i h^{(i)}\|,\end{equation}
where both norms in the definition are the standard Euclidean norms in $\R^{N_i}$.

\subsection{Fully parallel coordinate descent method}

PCDM used with an $n$-nice sampling $\hat{S}$ resembles DQAM in two ways:  i) it
updates \emph{all} blocks during each iteration, ii) it is not randomized.  Indeed, \begin{equation}\label{eq:s987a9s8}\sum_{i \in \hat{S}} U_i h^{(i)} = \sum_{i=1}^n U_i h^{(i)} = h,\end{equation}
and hence
\begin{eqnarray*}F(x+h) &\overset{\eqref{eq:s987a9s8}}{=}& \Exp \left[F(x+ {\textstyle \sum_{i \in \hat{S}} U_i h^{(i)}})\right] \quad \overset{\eqref{eq:09s0asjdxxXX}}{\leq} \quad F^{\text{ESO}}(x+h) \\
&\overset{\eqref{eq:09s0asjdxxXX}+ \eqref{H}}{=}& f(x) + \ve{f'(x)}{h} + \tfrac{\beta}{2}\|h\|_w^2 + \Psi(x+h).
\end{eqnarray*}

In particular, in the setting of Theorem~\ref{T_ESO_omega} we have $\beta = \omega$ and $w=L=(L_1,\dots,L_n)$, and Algorithm~\ref{A_PCD_ESO} specializes to Algorithm~\ref{A_PCD_DSO}.

\begin{algorithm}[H]
\caption{(Fully Parallel Coordinate Descent Method)}\label{A_PCD_DSO}
  \begin{algorithmic}[1]
  \State \textbf{Initialization:} $x_0 \in \R^N$
    \For {$k = 0,1,2,\dots$}
    \State \textbf{Step 1a:} Solve
    \begin{subequations}
\begin{equation}
\label{PCDM_step1098}
     h_k \gets \arg \min_{h \in \R^N} \left\{ f(x_k) + \ve{f'(x_k)}{ h} + \frac{\omega}{2}\sum_{i=1}^n  \ve{L_i B_i h^{(i)}}{h^{(i)}}+ \Psi(x_k+h) \right\}
\end{equation}
\State \textbf{Step 1b:} Update
\begin{equation}
\label{PCDM_step1098}
     x_{k+1} \gets x_k + h_k
\end{equation}
\end{subequations}
\EndFor
  \end{algorithmic}
\end{algorithm}

\section{Links Between DQAM and PCDM} \label{S_Comparison}

In this section we discuss and compare DQAM and PCDM. We highlight some of the main differences between the two methods, and describe a special case where the methods coincide.

\subsection{Fully parallel vs partially parallel updating}
\label{Ssub_parallel}
One of the main differences between DQAM and PCDM is the number of blocks that must be updated at each iteration. At each iteration of DQAM, \emph{all} $n$ blocks must be updated. This highlights the fact that DQAM uses a \emph{fully parallel} update scheme. On the other hand, PCDM is more flexible as it is able to update $\tau$ blocks at each iteration where $1\leq \tau \leq n$. This is beneficial because in practice there are usually fewer processors than the number of blocks. So, PCDM can act as a \emph{serial} method if $\tau = 1$, a \emph{fully parallel} method if $\tau = n$, or it can be optimized to the number of processors $p$ (so $\tau = p$). The advantages of updating $\tau = p$ blocks at each iteration of PCDM is established theoretically in Section~\ref{S_Complexity}.

Because DQAM updates all $n$ blocks at each iteration, it is a Jacobi type method, whereas PCDM can be interpreted as a Jacobi type method when $\tau = n$, a Gauss-Seidel type method when $\tau = 1$, or a hybrid Jacobi-Gauss-Seidel method  when $1 < \tau < n$.

\subsection{Flexibility of PCDM}

PCDM can be applied to a general convex composite function. Specifically, $f$ is only assumed to be smooth and convex. Further, the algorithm is guaranteed to converge when applied to a general smooth convex function, and can be equipped with iteration complexity bounds (see \cite{Richtarik12a}). On the other hand, the convergence results for DQAM have been only derived under the assumption that $f$ is quadratic and strongly convex; there are no convergence guarantees for a function $f$ with any other structure. Complexity estimates for both methods are discussed in detail in Section~\ref{S_Complexity}.

Notice that DQAM has been tailored specifically for an augmented Lagrangian objective function so it is reasonable that the function $f$ is assumed to be quadratic and strongly convex in this context. However, this assumption restricts the range of problems that can be solved using DQA, while PCDM can be applied to a much wider class of problems.

\subsection{Approximation type and algorithm philosophy}

In DQAM, a local two-sided approximation to the cross products is employed. The error associated with the approximation is of the order $o(\|h\|_2^2)$, which explains that, if the update $h_k$ is too large, then the model loses accuracy. This justifies the need for a correction step \eqref{DQA_step1c} so as to ensure that $x_{k+1}$ is not too far from $x_k$. This ensures a reduction in the objective value and ultimately, algorithm convergence. The need for a correction scheme within DQAM  is also apparent from the finite differences formulation presented in Algorithm~\ref{A_DQA-FD}. Consider the summation in \eqref{DQAFD_step1a}, and for simplicity assume that $\Psi \equiv 0$. Then the block update $h_k\ii$ is that which minimizes the function value difference in the $i$-th block coordinate direction, independently of all the other blocks $j \neq i$. Clearly, this will \emph{not} guarantee that $F(x_k + h_k) \leq F(x_k)$ because the function $F$ is not block separable. A simple 2D quadratic example showing that this approach is doomed to fail was described in
\cite{ChicagoICML13}.

In contrast to the DQAM scheme, PCDM employs a one-sided \emph{global} expected separable overapproximation of the augmented Lagrangian function \eqref{E_FaugLag}, which guarantees to produce a new random iterate $x_{k+1}$ that, on average,  decreases the objective function. That is, $x_{k+1}$ satisfies $\Exp [F(x_{k+1}) \;|\; x_k] \leq F(x_k)$. It turns out that this is sufficient to obtain a high probability complexity result and therefore there is no need for a correction step in PCDM. In fact, as we shall see in Section~\ref{S_97d070dsds}, a ``correction step'' is already embedded in the approximation in the form of the ESO parameter $\beta$.


Note that, besides DQAM, there are  many other algorithms that follow a ``step-then-correct'' strategy. One example are trust region methods, where a solution to some subproblem is found, the ``goodness'' of the solution is measured, and then the size of the trust region is adjusted to reflect the ``goodness''. A second example is the conditional gradient algorithm, which builds a linear approximation to the objective function, finds the minimizer of the linearized problem (the ``step'') and then ``corrects'' by taking a convex combination of the previous point and the step to reduce the objective value. This correction step is implicitly built-in for PCDM, in the choice of the constant $\beta$.

\subsection{A special case in which the methods coincide}\label{S_97d070dsds}

So far we have highlighted some of the differences between DQAM and PCDM. However, in this section we present a special case where the two methods coincide.

\begin{theorem}\label{T_coincide} Assume $f$ is  partially separable of degree $\omega$, and has block Lipschitz gradient with constants $L_1,L_2,\dots,L_n>0$. Further, assume $\Psi\equiv 0$. Then Algorithm~\ref{A_DQA-Hessian} (generalization of DQAM) coincides with Algorithm~\ref{A_PCD_DSO} (fully parallel PCDM) under the following choice of parameters:
\begin{equation}\label{eq:ds98dsnds}C_i(x_k) \equiv L_i B_i \quad (i=1,2,\dots,n), \qquad \theta = \tfrac{1}{\omega}.\end{equation}

\end{theorem}
\begin{proof} In Algorithm~\ref{A_DQA-Hessian} we have
$x_{k+1} = (1-\theta)x_k + \theta (x_k + h_k)$, where
\begin{equation}\label{eq:d989ds8d}h_k = \arg \min_{h\in \R^N} \{\ve{f'(x_k)}{h} + \tfrac{1}{2}\sum_{i=1}^n \ve{C_i(x_k) h^{(i)}}{h^{(i)}}\}.\end{equation}
Due to separability of the objective function in \eqref{eq:d989ds8d} and the choice of parameters \eqref{eq:ds98dsnds}, we see that $h_k^{(i)} = -\tfrac{1}{L_i}B_i^{-1}(f'(x_k))^{(i)}$, $i=1,2,\dots,n$, and hence
\begin{equation}\label{eq:09u0dsds}x_{k+1}^{(i)} = (1-\theta)x_k^{(i)} + \theta (x_k^{(i)} + h_k^{(i)}) = x_k^{(i)} -  \tfrac{1}{\omega L_i}B_i^{-1}(f'(x_k))^{(i)}.\end{equation}

In Algorithm~\ref{A_PCD_DSO} we have
$x_{k+1}= x_k + h_k$, where
 \begin{equation}\label{eq:d989adndc}h_k = \arg \min_{h\in \R^N} \left\{\ve{f'(x_k)}{h} + \frac{\omega}{2}\sum_{i=1}^n \ve{L_i B_i h^{(i)}}{h^{(i)}}\right\}.\end{equation}
Using separability of the objective function in \eqref{eq:d989adndc}, we again obtain the same formula \eqref{eq:09u0dsds} for $x_{k+1}$, establishing the equivalence of the two methods.
\end{proof}

A few remarks:

\begin{itemize}
\item
In the context of the original problem  \eqref{Eq_problem}, the case covered by the above theorem corresponds to a feasibility problem ($\Psi\equiv 0$ means that $g \equiv0$).

\item DQAM was analyzed in \cite{Ruszczynski95} only for the parameter $\theta$ in the interval $(0,\tfrac{1}{2(\omega-1)})$. For $\omega>1$ this leads to \emph{smaller steps} than the PCDM default choice $\theta = \tfrac{1}{\omega}$, which then translates to slower convergence for DQAM.
\end{itemize}

\section{Complexity of DQAM and PCDM under Strong Convexity} \label{S_Complexity}

In this section we study and compare the convergence rates of DQAM and PCDM under the assumption of strong convexity of the objective function. We limit ourselves to this case as complexity estimates for DQAM are not available otherwise. Both DQAM and PCDM benefit from linear convergence, but the rate is much better for PCDM than for DQA.

\paragraph{Strong convexity.} We assume that $F$ is strongly convex with respect to the norm $\| \cdot \|_w$ for some vector of positive weights $w=(w_1,\dots,w_n)$ specified in the results, with (strong) convexity parameter $\mu_F>0$. A function $\phi: \R^N \to \R \cup \{+ \infty\}$ is strongly convex with respect to the norm $\| \cdot \|_w$ with convexity parameter $\mu_\phi = \mu_{\phi}(w) \geq 0$ if for all $x,y \in \dom \phi$,
\begin{equation}
\label{strongly_convex_1}
     \phi(y) \geq \phi(x) + \langle \phi^{\prime}(x),y-x \rangle + \frac{\mu_{\phi}}{2}\|y-x\|_w^2,
\end{equation}
where $\phi^{\prime}(x)$ is any subgradient of $\phi$ at $x$. The case with $\mu_{\phi}(w)= 0$ reduces to convexity. It will be useful to note that for any $t>0$,
\begin{equation}\label{eq:d9jdshdsdkk}\mu_\phi(tw) = \frac{\mu_\phi(w)}{t}.\end{equation}

Strong convexity of $F$ may come from $f$ or $\Psi$ or both and we will write $\mu_f$ (resp. $\mu_{\Psi}$) for the strong convexity parameter of $f$ (resp. $\Psi$). It is easy to see that
\begin{equation}
     \label{strongly_convex_4}
\mu_F \geq \mu_f + \mu_\Psi.
\end{equation}

Note that the strong convexity constant of $F$ can be \emph{arbitrarily larger} than the sum of the strong convexity constants of the functions $f$ and $\Psi$. Indeed, consider the following simple 2D example ($N=n=2$): $f(x) = \tfrac{\mu}{2}(x^{(1)})^2$, $\Psi(x) = \tfrac{\mu}{2}(x^{(2)})^2$, where $\mu>0$. Let $\|x\|_w$ be the standard Euclidean norm (i.e., $B_i=1$ and $w_i=1$ for $i=1,2$). Clearly, neither $f$ nor $\Psi$ is strongly convex ($\mu_f=\mu_\Psi=0$). However, $F$ is strongly convex with constant $\mu_F=\mu$.

In the rest of the section we will repeatedly use the following simple result.

\begin{lemma} \label{eq:987a9s87a}Let $\xi_0>\epsilon>0$ and $\gamma\in (0,1)$. If $k\geq \tfrac{1}{\gamma}\log\left(\tfrac{\xi_0}{\epsilon}\right)$, then
$(1-\gamma)^k \xi_0 \leq \epsilon$.
\end{lemma}
\begin{proof}
$(1-\gamma)^k \xi_0 = (1-\tfrac{1}{1/\gamma})^{(1/\gamma)(\gamma k)} \xi_0 \leq e^{-\gamma k} \xi_0 \leq e^{-\log(\xi_0/\epsilon)}\xi_0 = \epsilon.$
\end{proof}

\subsection{PCDM}

We now derive a new improved complexity result for PCDM. In \cite[Theorem 20]{Richtarik12a}  the authors prove an iteration complexity bound based on the assumption that $\mu_f+\mu_\Psi>0$. Here we obtain a new and tighter complexity result under the weaker assumption  $\mu_F>0$. As discussed above, $\mu_F$ can be substantially bigger than $\mu_f+\mu_\Psi$, which implies that our complexity bound can be much better.

The following auxiliary result is an improvement on Lemma~17(ii) in \cite{Richtarik12a} and will be used in the proof of our main complexity result.

\begin{lemma}\label{eq:ds0nds6bsd}
     If $\mu_F(w) >0$ and $\beta \geq \mu_f(w)$, then for all $x \in\dom F$
\begin{equation}\label{eq:d9bnds87dsd}
     H_{\beta,w}(x+h(x)) -F^*\leq \frac{\beta - \mu_f(w)}{\mu_F(w)+\beta - \mu_f(w)}(F(x)-F^*).
\end{equation}
\end{lemma}
\begin{proof}
 Let $\mu_F = \mu_F(w)$, $\mu_f = \mu_f(w)$ and $\mu_\Omega = \mu_\Omega(w)$. By Lemma~16 in \cite{Richtarik12a}, we have
     \begin{equation} \label{eq:d09usdds}H_{\beta,w}(x+h(x)) \leq \min_{y \in \R^N} \left\{F(y) + \frac{\beta - \mu_f}{2}\|y-x\|_w^2\right\}.\end{equation}
Using this, we can further write
\begin{eqnarray}
     H_{\beta,w}(x+h(x)) &\overset{\eqref{eq:d09usdds}}{\leq} & \min_{y  = \lambda x^* + (1-\lambda)x,\; \lambda\in [0,1]} \left\{F(y) + \frac{\beta - \mu_f}{2}\|y-x\|_w^2\right\}\notag\\
&=& \min_{\lambda \in [0,1]} \left\{F(\lambda x^* + (1-\lambda)x) + \frac{(\beta - \mu_f)\lambda^2}{2}\|x-x^*\|_w^2\right\} \notag\\
&\leq& \min_{\lambda \in [0,1]} \left\{\lambda F^* + (1-\lambda)F(x) - \frac{\mu_F \lambda(1-\lambda)-(\beta - \mu_f)\lambda^2}{2}\|x-x^*\|_w^2\right\},\label{eq:09dsnsd8sd}
\end{eqnarray}
where in the last step we have used strong convexity of $F$. Notice that $\lambda^* \eqdef \mu_F/(\mu_F + \beta - \mu_f) \in (0,1]$ and that  $\mu_F(1-\lambda^*)-(\beta - \mu_f)\lambda^*=0$. It now only remains to substitute $\lambda^*$ into \eqref{eq:09dsnsd8sd} and subtract $F^*$ from the resulting inequality.
\end{proof}

We now present our main complexity result. It gives a bound on the number of iterations  required by PCDM (Algorithm~\ref{A_PCD_ESO}) to obtain an $\epsilon$ solution with high probability. The result is generic in the sense that it applies to any smooth convex function and proper uniform sampling as long as the parameters $\beta$ and $w$ giving rise to an ESO are known.

\begin{theorem}
\label{Thm_complexity}
     Assume that $F=f+\Psi$ is strongly convex with respect to the norm $\|\cdot\|_w$ ($\mu_F(w) >0$) and let $S_0,S_1,\dots$ be iid proper uniform samplings satisfying
     \[(f,S_0) \sim ESO(\beta,w).\]
     Choose an initial point $x_0 \in \R^N$, target confidence level $\rho \in (0,1)$, target accuracy level $0<\epsilon<F(x_0)-F^*$ and iteration counter
\begin{equation}
\label{D_K}
     K \geq \frac{n}{\Exp[|S_0|]}\frac{\beta + \mu_F(w) -\mu_f(w)}{\mu_F(w)}\log \left(\frac{F(x_0)-F^*}{\epsilon \rho}\right).
\end{equation}
 If $\{x_k\}$, $k \geq 0$, are the random points generated by PCDM (Algorithm~\ref{A_PCD_ESO}) as applied to problem \eqref{E_FaugLag}, then \[\Prob(F(x_K)-F^*\leq \epsilon)\geq 1-\rho.\]
\end{theorem}
\begin{proof}
Let $\alpha = \frac{\E[|S_0|]}{n}$ and $\xi_k = F(x_k) - F^*$. Then for all $k\geq 0$,
\begin{equation}
\label{E_sc}
     \E[\xi_{k+1}\;|\;x_k] \overset{\eqref{eq:09s0asjdxxXX} }{\leq} (1-\alpha)\xi_k + \alpha(H_{\beta,w}(x_k+h(x_k)) - F^*)\overset{\text{(Lemma~\ref{eq:ds0nds6bsd})}}{\leq} \Big(1 - \underbrace{\tfrac{\alpha\mu_F(w)}{\mu_F(w) + \beta-\mu_f(w)}}_{\eqdef \gamma}\Big)\xi_k.
\end{equation}
Note that Lemma~\ref{eq:ds0nds6bsd} is applicable as the assumption $\beta\geq \mu_f(w)$ is satisfied due to the fact that $(f,\hat{S})\sim ESO(\beta,w)$ (see \cite[Section~4]{Richtarik12a}). Further, note that $\gamma>0$ since $\alpha>0$ and $\mu_F(w)>0$. Moreover, $\gamma\leq 1$ since $\alpha \leq 1$ and $\beta \geq \mu_f(w)$.  By taking expectation in $x_k$ through \eqref{E_sc}, we obtain $\E[\xi_k]\leq (1-\gamma)^k \xi_0 $. Applying Markov inequality, Lemma~\ref{eq:987a9s87a} and \eqref{D_K}, we obtain
\begin{equation*}
     \Prob(\xi_K > \epsilon) \leq \frac{\E[\xi_K]}{\epsilon} \leq \frac{(1-\gamma)^K\xi_0}{\epsilon} \leq \rho,
\end{equation*}
establishing the result.
\end{proof}

In order to compare the complexity of PCDM with that of DQAM, which is a fully parallel method, we now derive a specialized complexity result for the fully parallel variant of PCDM (Algorithm~\ref{A_PCD_DSO}). The method is no longer stochastic in this situation, i.e., the sequence of vectors $\{x_k\}$, $k\geq 0$, is deterministic. Hence, we give a standard complexity result as opposed to a high probability one. Finally, we make use of the fact that for partially separable functions $f$, the parameters $\beta$ and $w$ are known.

\begin{theorem}\label{Thm_complexity2} Assume $f: \R^N\to \R$ is  partially separable of degree $\omega$, and has block Lipschitz gradient with constants $L_1,L_2,\dots,L_n>0$. Further assume that $F=f+\Psi$ is strongly convex with $\mu_F(L) >0$, where $L=(L_1,\dots,L_n)$. Finally, let $\{x_k\}_{k\geq 0}$ be the sequence generated by fully parallel PCDM (Algorithm~\ref{A_PCD_DSO}). Then for all $k\geq 0$,
\begin{equation}\label{eq:sa9hs8sd}
     F(x_{k+1})-F^* \leq q^{\text{\tiny PCDM}} (F(x_{k})-F^*),
\end{equation}
where
\begin{equation}
\label{E_q2}
     q^{\text{\tiny PCDM}} = 1-\frac{\mu_F(L)}{\omega + \mu_F(L)-\mu_f(L)}.
\end{equation}
Moreover, if we let $\epsilon< F(x_0)-F^*$ and \begin{equation}\label{eq:d9000sdsd}k\geq \frac{1}{1-q^{\text{\tiny PCDM}}} \log \left(\frac{F(x_0)-F^*}{\epsilon}\right),\end{equation} then $F(x_k)-F^* \leq \epsilon$.
\end{theorem}
\begin{proof} Let $\hat{S}$ be the fully parallel sampling, i.e., the $n$-nice sampling. Applying Theorem~\ref{T_ESO_omega}, we see that
$(f,\hat{S}) \sim ESO(\beta,w)$, with $\beta=\omega$ and $w=L$. Following the first part of the proof of Theorem~\ref{Thm_complexity}, we have $\alpha=1$ and
$\xi_{k+1} \leq (1-\gamma)\xi_k$, where $\gamma = \mu_F(L)/(\mu_F(L)+\omega-\mu_f(L))$, establishing \eqref{eq:sa9hs8sd}. The second statement follows directly by applying Lemma~\ref{eq:987a9s87a}.
\end{proof}

\subsection{DQAM}

We now present a complexity result for DQAM, established in \cite{Ruszczynski95}.

\begin{theorem}[Theorem 2 in \cite{Ruszczynski95}]
\label{T_DQArate} Let $f(x)=\tfrac{r}{2}\|b-Ax\|^2$ be partially separable of degree $\omega>1$. Assume that $F$ ($=f+\Psi$) is strongly convex with $\mu_F(e)>0$, where $e \in \R^n$ is the vector of all ones. Further assume that the sets $X_i$, $i = 1,\dots,n$, are bounded. Let $\{x_k\}$, $k\geq 0$, be the sequence generated by DQAM (Algorithm~\ref{A_DQA}) with $\theta = \tfrac{1}{2(\omega-1)}$. Then for all $k\geq 0$,
\begin{equation*}
    F(x_{k+1}) - F^* \leq q^{\text{\tiny DQAM}} \big(F(x_k)-F^* \big),
\end{equation*}
where
\begin{equation}
\label{E_q1}
    q^{\text{\tiny DQAM}} = 1-\frac{\mu_{F}(e)}{16 \lmax (\omega-1)^3 + 4(\omega-1)\mu_F(e)},
\end{equation}
and $\lmax \eqdef \max_{1\leq i \leq n }r\|A_i\|^2$. Moreover, if we let $\epsilon< F(x_0)-F^*$ and \begin{equation}\label{eq:d98usdsd}k\geq \frac{1}{1-q^{\text{\tiny DQAM}}} \log \left(\frac{F(x_0)-F^*}{\epsilon}\right),\end{equation} then $F(x_k)-F^* \leq \epsilon$.
\end{theorem}

Ruszczy\'{n}ski analyzed DQAM for a range of parameters $\theta$: $\theta \in (0,1/(\omega-1))$ \cite[Theorem 1; $\mu=0$]{Ruszczynski95}. However, the choice $\theta = 1/(2(\omega-1))$ is optimal \cite[Eq (5.11)]{Ruszczynski95}, and the above theorem presents Ruszczy\'{n}ski's result for this optimal choice of the stepsize parameter. A table translating the notation used in this paper and \cite{Ruszczynski95} is included in Appendix~B.

\subsection{Comparison of the Linear Rates of DQAM and PCDM} \label{S_complexity_comp}

We  now compare the convergence rates $q^{\text{\tiny DQAM}}$ and $q^{\text{\tiny PCDM}}$ defined in \eqref{E_q1} and \eqref{E_q2}, respectively, and the resulting iteration complexity guarantees. We will argue that $q^{\text{\tiny PCDM}}$ can be much better (i.e., smaller) than $q^{\text{\tiny DQAM}}$,  leading to vastly improved iteration complexity bounds. However, as we shall see, in practice the fully parallel PCDM method and DQAM behave similarly, with PCDM being about twice as fast as DQAM.

Before we start with the comparison, recall from \eqref{eq:08998as8s} that the gradient of $f(x)=\tfrac{r}{2}\|b-Ax\|^2$ (i.e., $f$ covered by Theorem~\ref{T_DQArate}) is block Lipschitz with constants
$L_i = r\|A_i^T A_i\|$, $i=1,2,\dots,n$. Hence, $L' = \max_i L_i$, which draws a link between the quantities $L_i$, $i=1,2,\dots,n$, appearing in Theorem~\ref{Thm_complexity2}  and $L'$ appearing in Theorem~\ref{T_DQArate}.

\begin{itemize}
\item \textbf{Identical Lipschitz constants.} Assume now that $L_i=L'$ for all $i=1,2,\dots,n$ and let $L = (L_1,\dots,L_n)$, as in Theorem~\ref{Thm_complexity2}.
Using \eqref{eq:d9jdshdsdkk} we observe that
\begin{equation}\label{eq:d09udn8}\mu_\phi(L) = \mu_\phi(L'e) = \frac{1}{L'}\mu_\phi(e),\end{equation}
whence
\begin{equation}\label{eq:ndkjsd67675}q^{\text{\tiny PCDM}}  \overset{\eqref{E_q2}+\eqref{eq:d09udn8}}{=} 1-\frac{\mu_F(e)}{L'\omega + \mu_F(e)-\mu_f(e)}.\end{equation}

We can now directly compare $q^{\text{\tiny PCDM}} $ and $q^{\text{\tiny DQAM}}$ by comparing \eqref{eq:ndkjsd67675} and \eqref{E_q1}. Clearly\footnote{This holds as long as $\omega>1$, which is the case covered by Theorem~\ref{T_DQArate} and hence assumed here.},
\begin{equation}\label{eq:jkdsihsd9} 16L'(\omega-1)^3\geq L'\omega \qquad \text{and} \qquad 4(\omega-1)\mu_F(e)\geq \mu_F(e)-\mu_f(e),\end{equation}
and hence $q^{\text{\tiny PCDM}} \leq q^{\text{\tiny DQAM}}$. However, both inequalities in \eqref{eq:jkdsihsd9} can be very loose, which means that $q^{\text{\tiny PCDM}}$ can be much better than $q^{\text{\tiny DQAM}}$.
For instance, in the case when $\mu_F(e)=\mu_f(e)$, we have
\begin{equation}\label{eq:0jds909sdXP}\frac{1-q^{\text{\tiny PCDM}} }{1-q^{\text{\tiny DQAM}} } = \frac{16L'(\omega-1)^3 + 4(\omega-1)\mu_F(e)}{L' \omega} \geq  \frac{16(\omega-1)^3}{\omega}.\end{equation}
In view of \eqref{eq:d9000sdsd} and \eqref{eq:d98usdsd}, this means that the number of DQAM iterations  needed to obtain an $\epsilon$-solution is larger than that for PCDM \emph{by at least the multiplicative factor} $16(\omega-1)^3/\omega$.  For instance, the theoretical iteration complexity of DQAM is more than 1000 times worse than that of PCDM for $\omega=10$.


\item\textbf{Varying Lipschitz Constants.} If the constants $L_1,\dots,L_n$ are not all equal, it is somewhat difficult to compare the complexity rates as we cannot directly compare the strong convexity constants $\mu_\phi(L)$ and $\mu_\phi(e)$ (for $\phi=F$ and $\phi=f$). What we can do, however, is to at least make sure that the ``scaling'' is identical in both. Here is what we mean by that. Recall that  $\mu_\phi(w)$ is the strong convexity constant of $\phi$ wrt a \emph{weighted norm} $\|x\|_w$ defined by \eqref{S_Norms_1}.  As we have remarked in \eqref{eq:d9jdshdsdkk}, if we scale the weights by a positive factor $t>0$, the corresponding strong convexity constant scales by $1/t$. Hence, $\mu_\phi(L)$ and $\mu_\phi(e)$ cannot be considered comparable unless  $\sum_i L_i = \sum_i e_i = n$. Of course, even if this was the case, it is possible that the strong convexity constants might be very different. However, in this case there is at least no reason to suspect a-priori that one might be larger than the other, and hence they are comparable in that sense.

    If we let $\bar{L} = \tfrac{1}{n}\sum_i L_i$ and $w_i = L_i/\bar{L}$ for $i=1,2,\dots,n$, then $\sum_i w_i = n$, and hence, as explained above,
    \begin{equation}\label{eq:09u9adnnmn}\mu_\phi(w)\approx \mu_\phi(e).\end{equation} Furthermore, since $w=L/\bar{L}$, we have
    \[\mu_\phi(L)= \mu_\phi\left( \bar{L}w\right) \overset{\eqref{eq:d9jdshdsdkk}}{=} \frac{1}{\bar{L}} \mu_\phi(w) \approx \frac{1}{\bar{L}} \mu_\phi(e).\]

The above is an analogue of \eqref{eq:d09udn8} and we can therefore now continue our comparison in the same way as we did for the case with identical Lipschitz constants. In particular, if $\mu_F(e) = \mu_f(e)$ (for simplicity), then as above we can argue that
\begin{equation}\label{eq:0jds909sdXP08098}\frac{1-q^{\text{\tiny PCDM}} }{1-q^{\text{\tiny DQAM}} } = \frac{16L'(\omega-1)^3 + 4(\omega-1)\mu_F(e)}{\bar{L} \omega} \geq  \frac{16(\omega-1)^3}{\omega}\frac{L'}{\bar{L}}.\end{equation}
Therefore, PCDM has an  even more dramatic theoretical advantage compared to DQAM in the case when the maximum Lipschitz constant $L'$ is much larger than the average $\bar{L}$.

\end{itemize}

\subsection{Optimal number of block updates} \label{S_OPT}

In this section we propose a simplified model of parallel computing and in it study the
performance of a family of parallel coordinate descent methods parameterized by a single parameter:
the number of blocks being updated in a single iteration.


In particular, consider the family of PCDMs where $S_k$ is a $\tau$-nice sampling and $\tau\in \{1,2,\dots,n\}$. Now assume we have $p \in \{1,2,\dots,n\}$ processors/threads available, each able to compute and apply to the current iterate the update $h^{(i)}(x_k)$ for a single block $i$, in a unit of time. PCDM, as analyzed, is a synchronous method. That is, a new parallel iteration can only start
once the previous one is finished, and hence updating $\tau$ blocks will take $\lceil \tfrac{\tau}{p}\rceil$ amount of time. On the other hand,
 the iteration complexity of PCDM is better for higher $\tau$. Indeed, by Theorem~\ref{T_ESO_omega},  $f$ satisfies an ESO with respect to $\hat{S}$ with parameters $w = L= (L_1,\dots,L_n)$ and $\beta = \beta(\tau) = 1+ \tfrac{(\omega-1)(\tau-1)}{n-1}$,
where $\omega$ is degree of partial separability of $f$ (we assume $n>1$). If, moreover, $\mu_F(L) = \mu_f(L)$, which is often the case as $\Psi$ is often not strongly convex, then Theorem~\ref{Thm_complexity} says that PCDM needs $\tfrac{ n}{\tau}\beta(\tau)c$ iterations, where $c$ is a constant independent of $\tau$, to solve \eqref{E_FaugLag} with high probability. Hence, the total amount of time needed for PCDM to solve the problem is equal to
\[T(\tau) = \lceil \tfrac{\tau}{p}\rceil \tfrac{ n}{\tau}\beta(\tau) c.\]

We can now ask the following natural question: what $\tau \in \{1,2,\dots,n\}$ minimizes $T(\tau)$? We now show that the answer is $\tau = p$.

\begin{theorem} \label{T_opt} Assume $f: \R^N\to \R$ is convex, partially separable of degree $\omega$, and has block Lipschitz gradient with constants $L_1,L_2,\dots,L_n>0$, where $n>1$. Further assume $\mu_F(L) = \mu_f(L) >0$ and consider the family of parallel coordinate descent methods with $\tau$-nice sampling, where $\tau\in \{1,2,\dots,n\}$, applied to problem \eqref{E_FaugLag}. Under the parallel computing model with $p \in \{1,2,\dots,n\}$ processors described above, the method with $\tau=p$ is optimal.
\end{theorem}
\begin{proof} We only need to show that
\[p = \arg \min \{T(\tau)\;:\; \tau = 1,2,\dots,n\}.\]
 It is easy to see that $\tfrac{n}{\tau}\beta(\tau)$ is decreasing in $\tau$. Since $\lceil \tfrac{\tau}{p}\rceil$ is constant for $kp+1\leq \tau \leq kp$, it suffices to consider $\tau=kp$ for $k=1,2,\dots$ only. Finally, $T(kp) = \tfrac{n}{p}\beta(kp)c$ is increasing in $k$ since $\beta(\cdot)$ is increasing, and we conclude that $k=1$ and hence $\tau=p$ is optimal.
\end{proof}

\section{Numerical Results} \label{S_Numerical}

In this section we present two numerical experiments that support the findings of this paper. In both experiments we choose $f(x) = \frac{1}{2}\|b-Ax\|^2$ and $\Psi \equiv 0$.

The first experiment considers the performance of DQAM and the fully parallel variant of PCDM in the above setting where we know that the two methods coincide up to he selection of the stepsize parameters $\omega$ and $\theta$ (recall Section~\ref{S_97d070dsds}). Here we focus on  comparing the effects of using the DQAM stepsize $\theta = 1/(2(\omega-1))$ versus the larger  PCDM stepsize $\theta = 1/\omega$.

The second experiment compares DQAM, fully parallel variant of PCDM (i.e., PCDM used with $n$-nice sampling) and PCDM used with $\tau$-nice sampling, in the situation when the number of available processors is $\tau$, while varying $\omega$ (degree of partial separability of $f$) and $\tau$.


\subsection{Impact of the different stepsizes of DQAM and PCDM}


 Suppose that $A$ has primal block angular structure
\begin{equation*}
  A = \begin{bmatrix}
    C \\ D
  \end{bmatrix} = \begin{bmatrix}
    C_1 &  & \\
    & \ddots & \\
    & &  C_n\\
    D_1 & \dots  & D_n
  \end{bmatrix},
\end{equation*}
where $C_i, D_i$ are matrices of appropriate sizes. Notice that when $D = 0$, the problem is partially  separable of degree $\omega = 1$ (i.e., it is fully separable) with respect to the natural block structure (i.e., blocks corresponding to the column submatrices $[C_i; 0; D_i]$). If $D$ is completely dense, the problem is nonseparable ($\omega = n$). In general, the degree of separability of $f$ is equal to the number of matrices $D_i$ that contain at least one nonzero entry.

In this (small scale) experiment we set $n = 100$ and let $C_1,\dots,C_{100}$ be 10\% dense matrices of size $150 \times 100$. Subsequently, $A$ is a $15,001 \times 10,000$ sparse matrix. The degree of separability of $f$ varies, and is controlled by setting a subset of the matrices $D_1,\dots,D_n$ to zero.

Twenty five random pairs $(A,b)$ were generated for each $\omega \in \{2,4,8,16,32\}$, and DQAM and fully parallel variant of DQAM were applied to each problem instance. A stopping condition of $f(x) \leq 10^{-4}b^T b$ was employed; the results of this experiment are presented in Figure~\ref{Stepsize:pdf}. All data points are averages over 25 runs.

\begin{figure}[h!]\centering
  \includegraphics[width=10cm]{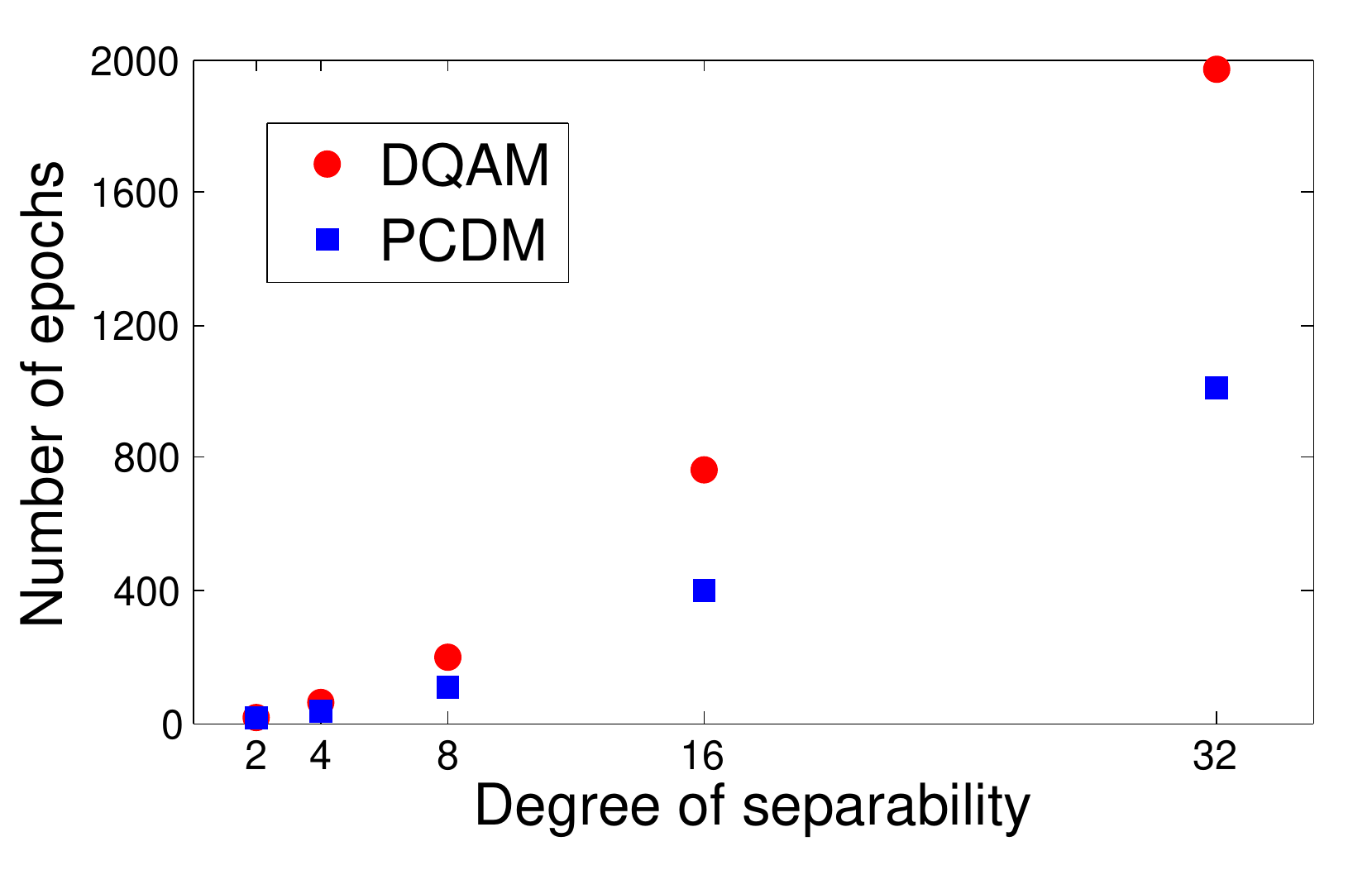}
   \label{Stepsize:pdf}
  \caption{This plot shows the number of epochs (a full sweep through the data, i.e., all $i=1,\dots,n$ blocks of $x$ are updated in one epoch) needed to solve the problem as a function of the degree of separability $\omega$. }

\end{figure}

Notice that when $\omega=2$, DQAM and PCDM require the same number of epochs to solve the problem. This is because $\theta= 1/(2(\omega-1)) = 1/2 = 1/\omega$. Then as $\omega$ grows, PCDM performs far better than DQAM, requiring almost 50\% fewer epochs than DQAM.

\subsection{Comparison of full vs partial parallelization}
\label{ExperimentTau}

Recall that unlike DQAM, PCDM is able to update $\tau$ blocks at each iteration, for any $\tau$ in the set $\{1,2,\dots,n\}$, demonstrating  useful  flexibility of the algorithm. By PCDM($\tau$) we denote the variant of PCDM in which $\tau$ blocks are updated at each iteration, using a $\tau$-nice sampling. In this experiment we investigate the performance of DQAM, PCDM($n$) (which in the plots we refer to simply as PCDM) and PCDM($\tau$), for a selection of parameters $\tau$ (the number of processors), and $\omega$ (the degree of partial separability).

Let us call the time taken for all $\tau$ processors to update a single block, one ``time unit''. Then, after one time unit of PCDM($\tau$), new gradient information is available to be utilized during the next time unit, which is much earlier than if all $n$ blocks need to be updated in each iteration.
On the other hand, for DQAM and PCDM, one iteration corresponds to all $n$ blocks of $x$ being updated. Subsequently, if there are $\tau$ processors available, one iteration of DQAM or PCDM (one epoch) corresponds to $\lceil \frac{n}{\tau}\rceil$ time units. However, PCDM($\tau$) will need to perform more iterations than both DQAM and PCDM. When both of these factors are taken into account, we have shown in Theorem~??? that PCDM($\tau$) is optimal in terms of overall complexity if there are $\tau$ processors.

The purpose of this experiment is to investigate this phenomenon numerically. Further, let $A$ be a $2 \cdot 10^4 \times 10^4$ sparse matrix, with at most $\omega$ nonzero entries per row. Let the stopping condition be $f(x) \leq 10^{-4}b^Tb$. The experiment was run for three instances: $\omega = 20,60,100$, and for each $\omega$ and varying $\tau$, the average number of time units required by DQAM, PCDM and PCDM($\tau$) were recorded. The results are shown in Figure~\ref{Timeunits}.

\begin{figure}[h!]\centering
  \includegraphics[width=12cm]{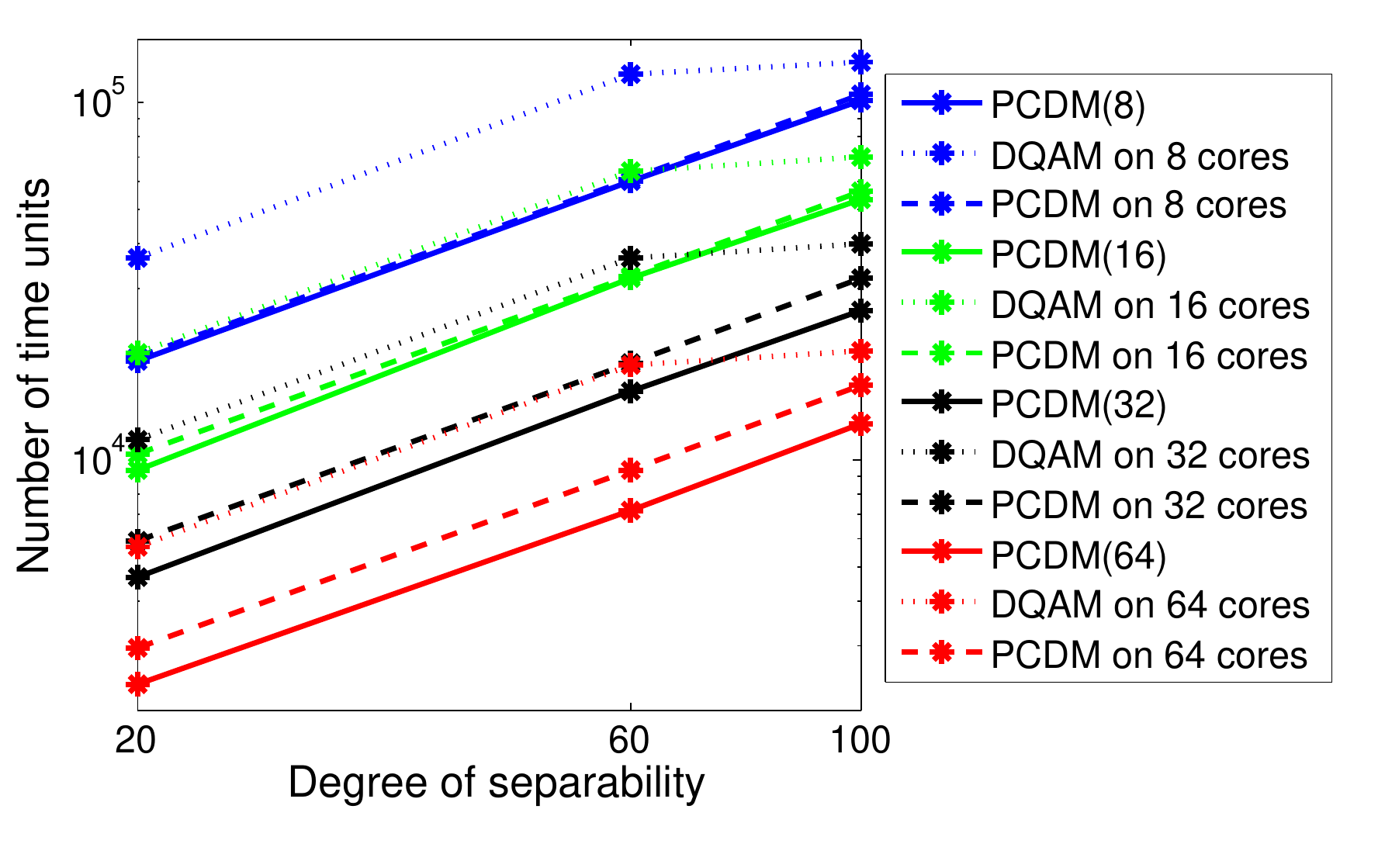}
   \caption{For each fixed $\tau\in \{8,16,32,64\}$, PCDM($\tau$) (solid line) is better than PCDM (dashed line), and noth are far better than DQAM (dotted line).}
  \label{Timeunits}
\end{figure}

The colors in Figure~\ref{Timeunits} correspond to different values of $\tau$. The solid lines correspond to PCDM($\tau$), while the dotted line (respectively dashed line) corresponds to DQAM (respectively PCDM) run with $\tau$ processors available. As $\omega$ increases, all algorithms require a higher number of time units. Further, as the number of available processors increases, the number of time units decreases. More importantly, for any fixed $\tau$, PCDM($\tau$), requires  far fewer time units than PCDM, and both require many fewer time units than DQAM. (Notice the log scale.) This demonstrates the practical advantage of `optimizing' PCDM($\tau$) to the number of available processors, as described in Section~7.4.

We have also recorded the average cpu time, and the resulting curves are visually indistinguishable from those in Figure~\ref{Timeunits}; only the scale of the vertical axis changes.

%
%
%
%

\bibliography{ref}

\begin{thebibliography}{10}

\bibitem{Berger94}
Arno~J. Berger, John~M. Mulvey, and Andrzej Ruszczy\'{n}ski.
\newblock An extension of the {DQA} algorithm to convex stochastic programs.
\newblock {\em SIAM Journal on Optimization}, 4(4):735--753, November 1994.

\bibitem{Bertsekas96}
Dimitri Bertsekas.
\newblock {\em Constrained Optimization and Lagrange Multiplier Methods}.
\newblock Athena Scientific, 1996.

\bibitem{Blondel2013}
Mathieu Blondel, Kazuhiro Seki, and Kuniaki Uehara.
\newblock Block coordinate descent algorithms for large-scale sparse multiclass
  classification.
\newblock {\em Machine Learning}, 2013.

\bibitem{Bradley:PCD-paper}
Joseph~K. Bradley, Aapo Kyrola, Danny Bickson, and Carlos Guestrin.
\newblock Parallel coordinate descent for {L1}-regularized loss minimization.
\newblock In {\em 28th International Conference on Machine Learning}, 2011.

\bibitem{Fercoq:Adaboost-2013}
Olivier Fercoq.
\newblock Parallel coordinate descent for the {A}daboost problem.
\newblock Technical report, July 2013.

\bibitem{Fercoq-Richtarik:SmoothedPCDM-2013}
Olivier Fercoq and Peter Richt\'{a}rik.
\newblock Smoothed parallel coordinate descent method.
\newblock Technical report, 2013.

\bibitem{hestenes69}
Magnus~R. Hestenes.
\newblock Multiplier and gradient methods.
\newblock {\em Journal of Optimization Theory and Applications}, 4:303--320,
  1969.

\bibitem{Lin:2008:DCDM}
Cho-Jui Hsieh, Kai-Wei Chang, Chih-Jen Lin, S~Sathiya Keerthi, and
  S~Sundararajan.
\newblock A dual coordinate descent method for large-scale linear svm.
\newblock In {\em ICML 2008}, pages 408--415, 2008.

\bibitem{2MIT2013}
Yin~Tat Lee and Aaron Sidford.
\newblock Effcient accelerated coordinate descent methods and faster algorithms
  for solving linear systems.
\newblock {\em arXiv:1305:1922v1}, 2013.

\bibitem{Li:CDOMACSGA}
Yingying Li and Stanley Osher.
\newblock Coordinate descent optimization for $l_1$ minimization with
  application to compressed sensing; a greedy algorithm.
\newblock {\em Inverse Problems and Imaging}, 3:487--503, August 2009.

\bibitem{Lu-Xiao2013}
Zhaosong Lu and Lin Xiao.
\newblock On the complexity analysis of randomized block-coordinate descent
  methods.
\newblock Technical report, May 2013.
\newblock arXiv:1305.4723.

\bibitem{Lu-Xiao2013b}
Zhaosong Lu and Lin Xiao.
\newblock Randomized block coordinate non-monotone gradient method for a class
  of nonlinear programming.
\newblock Technical report, June 2013.
\newblock arXiv:1306.5918.

\bibitem{Mulvey92}
John~M. Mulvey and Andrzej Ruszczy\'{n}ski.
\newblock A diagonal quadratic approximation method for large scale linear
  programs.
\newblock {\em Operations Research Letters}, 12:205--215, 1992.

\bibitem{Mulvey95}
John~M. Mulvey and Andrzej Ruszczy\'{n}ski.
\newblock A new scenario decomposition method for large scale stochastic
  optimization.
\newblock {\em Operations Research}, 43(3):477--490, 1995.

\bibitem{Necoara12}
Ion Necoara, Yurii Nesterov, and Francois Glineur.
\newblock Efficiency of randomized coordinate descent methods on optimization
  problems with linearly coupled constraints.
\newblock Technical report, June 2012.

\bibitem{Necoara_composite-coupled}
Ion Necoara and Andrei Patrascu.
\newblock A random coordinate descent algorithm for optimization problems with
  composite objective function and linear coupled constraints.
\newblock Technical report, University Politehnica Bucharest, 2012.
\newblock arXiv:1302.3074.

\bibitem{Nesterov12}
Yurii Nesterov.
\newblock Efficiency of coordinate descent methods on huge-scale optimization
  problems.
\newblock {\em SIAM Journal on Optimimization}, 22(2):341--362, 2012.

\bibitem{Necoara13-nonconvex}
Andrei Patrascu and Ion Necoara.
\newblock Efficient random coordinate descent algorithms for large-scale
  structured nonconvex optimization.
\newblock Technical report, University Politehnica Bucharest, May 2013.

\bibitem{powell72}
Michael J.~D. Powell.
\newblock A method for nonlinear constraints in minimization problems.
\newblock In Roger Fletcher, editor, {\em Optimization}, pages 283--298.
  Academic Press, 1972.

\bibitem{Qin10}
Zhiwei~(Tony) Qin, Katya Scheinberg, and Donald Goldfarb.
\newblock Efficient block-coordinate descent algorithms for the group lasso.
\newblock Technical report, Department of Industrial Engineering and Operations
  Research, Columbia University, 2010.

\bibitem{RT:TTD2011}
Peter Richt\'{a}rik and Martin Tak\'{a}\v{c}.
\newblock Efficient serial and parallel coordinate descent methods for
  huge-scale truss topology design.
\newblock In {\em Operations Research Proceedings 2011}, pages 27--32.
  Springer, 2012.

\bibitem{Richtarik12}
Peter Richt\'{a}rik and Martin Tak\'{a}\v{c}.
\newblock Iteration complexity of randomized block-coordinate descent methods
  for minimizing a composite function.
\newblock {\em Mathematical Programming, Ser. A}, 2012.

\bibitem{Richtarik12a}
Peter Richt\'{a}rik and Martin Tak\'{a}\v{c}.
\newblock Parallel coordinate descent methods for big data optimization.
\newblock Technical report, November 2012.
\newblock arXiv:1212.0873.

\bibitem{RT:SPARS11a}
Peter Richt\'{a}rik and Martin Tak\'{a}\v{c}.
\newblock Efficiency of randomized coordinate descent methods on minimization
  problems with a composite objective function.
\newblock In {\em 4th Workshop on Signal Processing with Adaptive Sparse
  Structured Representations}, June 2011.

\bibitem{rockafellar73}
R.~Tyrell Rockafellar.
\newblock The multiplier method of {Hestenes} and {Powell} applied to convex
  programming.
\newblock {\em Journal of Optimization Theory and Applications}, 12:555--562,
  1973.

\bibitem{rockafellar76}
R.~Tyrell Rockafellar.
\newblock Augmented {Lagrangians} and applications of the proximal point
  algorithm in convex programming.
\newblock {\em Mathematics of Operations Research}, 1:97--116, 1976.

\bibitem{Rockafellar91}
R.~Tyrell Rockafellar and Roger J.-B. Wets.
\newblock Scenarios and policy aggregation in optimization under uncertainty.
\newblock {\em Mathematics of Operations Research}, 16:1--23, 1991.

\bibitem{Ruszczynski89}
Andrzej Ruszczy\'{n}ski.
\newblock An augmented {L}agrangian method for block diagonal linear
  programming problems.
\newblock {\em Operations Research Letters}, 8:287--294, 1989.

\bibitem{Ruszczynski95}
Andrzej Ruszczy\'{n}ski.
\newblock On convergence of an augmented {L}agrangian decomposition method for
  sparse convex optimization.
\newblock {\em Mathematics of Operations Reseach}, 20(3):634--656, 1995.

\bibitem{Schwarz1870}
Hermann Schwarz.
\newblock \"{U}ber einen {G}renz\"{u}bergang durch alternierendes {V}erfahren.
\newblock {\em Vierteljahrsschrift der Naturforschenden Gesellschaft in
  Z\"{u}rich}, 15:272--286, 1870.

\bibitem{ShalevTewari09}
Shai Shalev-Shwartz and Ambuj Tewari.
\newblock Stochastic methods for $l_1$ regularized loss minimization.
\newblock In {\em 26th International Conference on Machine Learning}, 2009.

\bibitem{SSS2013-accelerated}
Shai Shalev-Shwartz and Tong Zhang.
\newblock Accelerated mini-batch stochastic dual coordinate ascent.
\newblock Technical report, May 2013.
\newblock arXiv:1305.2581.

\bibitem{SSS2013}
Shai Shalev-Shwartz and Tong Zhang.
\newblock Stochastic dual coordinate ascent methods for regularized loss
  minimization.
\newblock {\em Journal of Machine Learning Research}, 14:567--599, 2013.

\bibitem{stepwest75}
George Stephanopoulos and Arthur~W. Westerberg.
\newblock The use of {Hestenes'} method of multipliers to resolve dual gaps in
  engineering system optimization.
\newblock {\em Journal of Optimization Theory and Applications}, 15:285--309,
  1975.

\bibitem{ChicagoICML13}
Martin Tak\'{a}\v{c}, Avleen Bijral, Peter Richt\'{a}rik, and Nathan Srebro.
\newblock Mini-batch primal and dual methods for {SVM}s.
\newblock In {\em 30th International Conference on Machine Learning}, 2013.

\bibitem{TRG:Inexact2013}
Rachael Tappenden, Peter Richt\'{a}rik, and Jacek Gondzio.
\newblock Inexact coordinate descent: complexity and preconditioning.
\newblock Technical report, April 2013.
\newblock arXiv:1304.5530.

\bibitem{Tseng01}
Paul Tseng.
\newblock Convergence of a block coordinate descent method for
  nondifferentiable minimization.
\newblock {\em Journal of Optimization Theory and Applications}, 109:475--494,
  June 2001.

\bibitem{watnismat78}
N.~Watanabe, Y.~Nishimura, and M.~Matsubara.
\newblock Decomposition in large system optimization using the method of
  multipliers.
\newblock {\em Journal of Optimization Theory and Applications}, 25:181--193,
  1978.

\end{thebibliography}

\newpage
\appendix

\section{Notation Dictionary}

For the reader interested in comparing our work with the paper \cite{Ruszczynski95} directly, we have included a brief dictionary translating some of the key notation (Table~1).

\begin{table}[!h]\label{table:notation}
\caption{Notation dictionary.}
\begin{center}
\begin{tabular}{|c | c| }
     \hline
     Ruszczy\'{n}ski \cite{Ruszczynski95} & This paper\\
     \hline
     $L$ & $n$\\
     $N$ & $\omega-1$\\
     $x_i$ & $x^{(i)}$\\
     $ \tilde{x}$ & $x$\\
     $x$ & $y$\\
     $x - \tilde{x}$ & $h = y-x$\\
     $\tau $& $\theta$\\
     $\rho $& $r$\\
     $\rho\alpha^2$ & $L'$\\
     $\gamma$ & $\mu_F(e)/2$\\
     $\frac{1}{2}r\|b - \sum_{i=1}^n A_ix_i\|_2^2$ & $f(x)$\\
     $f_i(x_i)-\langle A_i^T\pi, x_i \rangle$ & $\Psi_i(\xbi) \;\;(= g_i(x_i)-\langle A_i^T\pi, x_i \rangle)$\\
     $\Lambda(x)$ & $F(x) = f(x) + \Psi(x)$\\
     $\Lambda_i(x_i,\tilde{x})$ & $f(x+U_ih^{(i)}) + \Psi_i(y^{(i)})$ \\
     $\tilde{\Lambda}(x,\tilde{x})$ & $f(x) + \sum_{i=1}^n [f(x+U_ih^{(i)}) -f(x) ] + \Psi(x+h) $ \\
     \hline
\end{tabular}
\end{center}
\end{table}

\end{document}